\documentclass[10pt,reqno]{amsart}

\usepackage{verbatim}
\usepackage{graphicx}
\usepackage{amssymb}
\usepackage{hyperref}
\usepackage{enumerate}
\usepackage{enumitem}
\usepackage{multicol}

\usepackage{color}

\renewcommand{\bar}{\overline}

\renewcommand{\tilde}{\widetilde}

\renewcommand{\implies}{\Rightarrow}

\usepackage{setspace}

\usepackage{mathrsfs}

\newcommand{\gNorm}[1]{{\left\|\kern-0.24ex\left|{ #1 } \right|\kern-0.24ex\right\|}}

\newcommand{\Nabla}[1]{ {^{(#1)}{\nabla}} }

\renewcommand{\bar}{\overline}

\renewcommand{\epsilon}{\varepsilon}

\newcommand{\D}{d}

\newcommand{\B}{\mathcal{H}}

\newcommand{\phg}{\textup{\normalfont phg}}
\newcommand{\WAH}{\mathscr{M}_\textup{weak}}
\newcommand{\HAR}{\mathscr{M}_\textup{R}}

\renewcommand{\subset}{\subseteq}
\renewcommand{\setminus}{\!\smallsetminus\!}

\renewcommand{\(}{\textup{(}}
\renewcommand{\)}{\textup{)}}

\newcommand{\Defn}[1]{{\boldmath\it\bfseries #1}}

\DeclareMathOperator{\Riem}{Riem}
\DeclareMathOperator{\Ric}{Ric}
\DeclareMathOperator{\R}{R}
\DeclareMathOperator{\grad}{grad}
\DeclareMathOperator{\Hess}{Hess}

\DeclareMathOperator{\Div}{div}
\DeclareMathOperator{\tr}{tr}

\DeclareMathOperator{\Id}{id}

\theoremstyle{plain}
\newtheorem{theorem}{Theorem}
\newtheorem{lemma}[theorem]{Lemma}
\newtheorem{proposition}[theorem]{Proposition}
\newtheorem{corollary}[theorem]{Corollary}
\newtheorem{defn}[theorem]{Definition}
\newtheorem{remark}[theorem]{Remark}

\numberwithin{theorem}{section}
\numberwithin{equation}{section}

\title[Shear-free hyperboloidal data]{The shear-free condition and constant-mean-curvature hyperboloidal initial data}
\author{Paul T.~Allen, James Isenberg, John M.~Lee, Iva Stavrov Allen}

\email{ptallen@lclark.edu}

\email{johnmlee@uw.edu}

\email{isenberg@uoregon.edu}

\email{istavrov@lclark.edu}

\address{Department of Mathematical Sciences, Lewis \& Clark College}

\address{Department of Mathematics, University of Washington}

\address{Department of Mathematics, University of Oregon}

\address{Department of Mathematical Sciences, Lewis \& Clark College}

\subjclass[2010]{Primary 35Q75; Secondary 53C80, 83C05}

\begin{document}

\begin{abstract}
We consider the Einstein--Maxwell--fluid constraint equations, and make use of the conformal method to 
construct and parametrize constant-mean-curvature hyperboloidal initial data sets that satisfy the shear-free condition. 
This condition is known to be necessary in order that a spacetime development admit a regular conformal boundary at future null infinity; see \cite{AnderssonChrusciel-Obstructions}.
We work with initial data sets in a variety of regularity classes, primarily considering those data sets whose geometries are \emph{weakly asymptotically hyperbolic}, as defined in \cite{WAH-Preliminary}.
These metrics are $C^{1,1}$ conformally compact, but not necessarily $C^2$ conformally compact.
In order to ensure that the data sets we construct are indeed shear-free, we make use of the conformally covariant traceless Hessian introduced in \cite{WAH-Preliminary}. 
We furthermore construct a class of initial data sets with weakly asymptotically hyerbolic metrics that may be only $C^{0,1}$ conformally compact; these data sets are insufficiently regular to make sense of the shear-free condition.
%
%
%
\end{abstract}

\maketitle

\setcounter{tocdepth}{2}


\section{Introduction}

Asymptotically flat spacetimes are used to model isolated astrophysical systems, and since the work of Penrose \cite{Penrose-AsymptoticBehavior}, it has been recognized  that one of the most useful mathematical ways to define and work with such spacetimes is to require that they admit a conformal compactification;  see, e.g.,~\cite{Frauendiener-ConformalInfinity}, \cite{HawkingEllis}, \cite{Wald}.
 In particular, if a  spacetime $(\mathbf M, \mathbf g)$ is asymptotically flat in the sense of admitting a conformal compactification\footnote{Historically, the term ``asymptotically simple" has been used for such spacetimes; informal usage calls these spacetimes asymptotically flat.}, 
then the manifold $\mathbf M$ is the interior of a closed manifold $\bar{\mathbf M}$ with boundary $\partial \mathbf M$, and the metric $\mathbf g$ can be written as $\mathbf g = \Omega^{-2}\bar{\mathbf g}$ for some  metric $\mathbf {\bar g}$ on $\mathbf {\bar M}$ and some non-negative function $\Omega\colon \mathbf{\bar M}\to[0,\infty)$, satisfying $\Omega^{-1}(0) = \partial \mathbf M$ and  $\D\Omega\neq 0$ along future and past null infinity $\mathscr I^+ \cup \mathscr I^- \subset \partial\mathbf M$.

We are interested in setting up the initial value problem as a tool for the construction and analysis of asymptotically flat spacetime solutions of Einstein's equations, making sure that the solutions being studied do indeed admit conformal compactifications. To understand how to do this, it is useful to consider foliations of a given asymptotically flat spacetime $(\mathbf M, \mathbf g)$  such that each leaf of the foliation is everywhere spacelike, and each leaf intersects the conformal  boundary $\partial {\mathbf M}$ along future null infinity $\mathscr I^+$. One leaf of such a foliation, together with its induced metric and its induced second fundamental form  (jointly satisfying the Einstein constraint equations), comprises an ``initial data set'' for the Einstein field equations.
While such leaves are not Cauchy surfaces for the entire spacetime, they suffice for the future evolution problem.
Furthermore, such foliations are natural for studying the outgoing gravitational and electromagnetic radiation of an isolated system and have been used in several numerical studies; see \cite{MoncriefRinne-Regularity},  \cite{Rinne-Axisymmetric}, \cite{MoncriefRinne-EinsteinYM}, \cite{Anil-HyperboloidalEvolution}, \cite{Anil-KerrTails}.

The simplest example of a spacetime foliated as above is the Minkowski spacetime $(\mathbf M=\mathbb R^{4},\mathbf g = -(\D x^0)^2 + (\D x^1)^2 + (\D x^2)^2 + (\D x^3)^2)$, foliated by the hyperboloids
\begin{equation}
M_t = \{-(x^0-t)^2 + (x^1)^2 + (x^2)^2 + (x^3)^2=-1\}.
\end{equation}
It is easy to see that the induced geometry  on $M_t$ is hyperbolic; indeed under an appropriate conformal compactification of the spacetime, the leaves $M_t$ are simply copies of the Poincar\'e disk model of hyperbolic space (see \cite{HawkingEllis},\cite{Wald}).
For more general asymptotically flat spacetimes, there is a wide class of foliations whose leaves $(M,g)$ intersect $\mathscr I^+$ transversely and are {asymptotically hyperbolic} 
in the following sense:
Let the manifold $M$ be the interior of a smooth compact manifold with boundary $\bar M$.
A $C^1$ function $\Omega\colon\bar M \to [0,\infty)$ is a \Defn{defining function} if $\Omega^{-1}(0) = \partial  M$ and $\D \Omega \neq 0$ along $\partial  M$.
A Riemannian metric $g$ on $M$ is called \Defn{conformally compact} if $g = \Omega^{-2}\bar g$ for some continuous metric $\bar g$ on $\bar M$ and some defining function $\Omega$.
If $\bar g$ is at least $C^2$ on $\bar M$, then the sectional curvatures of $g$ approach $-|\D\Omega|^2_{\bar g}$  as $\Omega\to 0$.
Thus we say that a conformally compact $(M,g)$ is
\Defn{\(strongly\) asymptotically hyperbolic} if $\bar g\in C^2(\bar M)$ and $|\D\Omega|_{\bar g}=1$ along $\partial  M$.

An initial data set for the Einstein equations on an asymptotically hyperbolic manifold, including a specification of the second fundamental form as well as perhaps certain non-gravitational fields (see below), is commonly called a ``hyperboloidal" data set in the literature. These have been studied in  
\cite{AnderssonChrusciel-Obstructions},
\cite{AnderssonChrusciel-Dissertationes},
\cite{AnderssonChruscielFriedrich},
\cite{GicquaudSakovich},
\cite{IsenbergLeeStavrov-AHP},
\cite{IsenbergPark}.
One finds that if a set of hyperboloidal data is to be used to generate an asymptotically flat spacetime with regular asymptotic conformal structure, then in addition to the usual constraint equations, the data must satisfy a boundary condition called the {\it shear-free condition} (discussed in Section \ref{Section-Preliminaries} below); see \cite{Frauendiener-ConformalInfinity}, \cite{MoncriefRinne-Regularity}, and especially \cite{AnderssonChrusciel-Obstructions}, where the issue of smooth conformal compactifications is studied in detail.

Because satisfying the shear-free condition is essential for a hyperboloidal initial data set to have any chance of generating a spacetime development that admits a conformal compactification at future null infinity, a more appropriate notion of ``hyperboloidal initial data" would include the shear-free condition as part of the definition. 
Unfortunately, many results concerning the existence of hyperboloidal initial data (\cite{AnderssonChrusciel-Dissertationes}, \cite{IsenbergPark}, \cite{GicquaudSakovich}, {et al.}) do not explicitly address the existence of data satisfying the shear-free condition.
In fact, Proposition 3.2 of \cite{AnderssonChrusciel-Obstructions} explicitly states that among those asymptotically hyperbolic solutions to the constraint equations constructed in \cite{AnderssonChrusciel-Dissertationes}  from smooth ``seed data'' (see the discussion in \S\ref{Section-ConformalMethod}), the shear-free condition is generically not satisfied.
(Note that the genericity result in \cite{AnderssonChrusciel-Obstructions} is with respect to the ``compactified'' $C^\infty(\bar M)$ topology; in \cite{SFDensity-Prelim} it is shown that shear-free data {\it is} dense in the ``physical'' $C^{k}(M)$ topology.)
One of our  purposes here is to clarify the existential status  of hyperboloidal data that does satisfy the shear-free condition.

In the present work we make a systematic study of hyperboloidal initial data sets satisfying the shear-free condition in the constant-mean-curvature (CMC) setting.
It is in particular among our goals to advertise and to clarify the role of the shear-free condition in the study of the Einstein constraint equations.
As part of this effort, we systematically incorporate the shear-free condition into the conformal method, obtaining a parametrization of the collection of  all sufficiently regular
 shear-free CMC hyperboloidal data sets satisfying the Einstein constraint equations.
As astrophysical systems typically contain matter fields as well as gravitational fields, we consider data sets that include electromagnetic and fluid source fields, with the data collectively satisfying the Einstein--Maxwell--fluid constraint equations.

A key tool in our analysis is the tensor $\B_{\bar g}(\rho)$, a conformally covariant version of the traceless Hessian (see \eqref{DefineB} below), which we have introduced in \cite{WAH-Preliminary}.
A hyperboloidal initial data set is shear-free if and only if the trace-free part of the second fundamental form agrees, to leading order, with $\B_{\bar g}(\rho)$ along $\partial M$.
The conformal covariance of $\B_{\bar g}(\rho)$ guarantees that this characterization of the shear-free condition is compatible with the conformal method, and can consequently be built directly into the choice of seed data, which we define below (see Section \ref{Section-ConformalMethod}).

Penrose's early studies \cite{Penrose-AsymptoticBehavior} of asymptotically flat spacetimes via conformal compactification required that the metric $\bar{\mathbf g}$ extend smoothly to $\mathscr I^\pm$.
Initial data sets with smooth conformal compactifications have been constructed via the conformal method in \cite{AnderssonChruscielFriedrich}.
However, large classes of hyperboloidal data constructed by the conformal method do not admit smooth compactifications---even if the ``seed data'' is smooth on $\bar M$, the resulting solution to the Einstein constraint equations may be polyhomogeneous, rather than smooth, along the boundary $\partial M$; see \cite{AnderssonChrusciel-Obstructions}, \cite{AnderssonChrusciel-Dissertationes}.
Roughly speaking, a tensor field $f$ on $M$ is polyhomogeneous if for any fixed smooth defining function $\rho$ on $\bar M$, the coordinate expression of $f$ has an asymptotic expansion at $\partial M$ in powers of $\rho$ and $\log\rho$.
Such fields are smooth on the interior $M$, but may be differentiable only to finite order on $\bar M$.
We refer the reader to the appendix of \cite{WAH-Preliminary} for additional details.

Motivated by the results of \cite{AnderssonChrusciel-Obstructions} and \cite{AnderssonChrusciel-Dissertationes}, we consider sets of seed data satisfying  weaker regularity conditions, namely that the metric satisfy the {\it weakly asymptotically hyperbolic} condition of \cite{WAH-Preliminary}.
The sectional curvatures of such metrics approach $-1$ at the conformal boundary but, in contrast to strongly asymptotically hyperbolic metrics, might only be $C^{1,1}$ conformally compact; for a detailed definition, see Section \ref{Section-DefineData}.
We are able to construct from such sets of seed data shear-free CMC solutions to the constraint equations that are weakly asymptotically hyperbolic.
Furthermore, if the seed data is polyhomogeneous, then the resulting shear-free data is also polyhomogeneous. 
Our methods may also be used to produce less regular asymptotically hyperbolic solutions to the constraint equations that, while not guaranteed to satisfy the shear-free condition, also need not be more than Lipschitz conformally compact; see Section \ref{Section-WeakExistence}.

Under the conformal method, the Einstein constraint equations reduce to a system of elliptic differential equations which, in the hyperboloidal setting, are uniformly degenerate.
Such operators have been systematically studied; see especially 
\cite{Mazzeo-Edge}, 
\cite{Melrose-TransformationMethods}, \cite{Melrose-Transformation}, \cite{Lee-FredholmOperators} and references therein.
We rely here on the results of \cite{WAH-Preliminary}, where the Fredholm results for elliptic operators are extended to the weakly asymptotically hyperbolic setting.

Our work below is organized as follows.
Following a discussion of the Einstein--Maxwell--fluid system and the shear-free condition, we carefully define several regularity classes of constant-mean-curvature shear-free hyperboloidal data sets in Section \ref{Section-DefineData}.
We next discuss the tensor $\B_{\bar g}(\rho)$ in Section \ref{Section-B}.
In \S\ref{Section-ConformalMethod} we outline the conformal method for constructing CMC hyperboloidal data sets.
We present in \S\ref{PDEResults} a number of PDE results that are needed to implement the conformal method; in \S\ref{PDE-Continuity} we show the continuity of these PDE results.
Then, in \S\ref{Section-PandE}, we carry out this implementation, obtaining the parametrization of hyperboloidal data by suitable seed data in Theorem \ref{ParametrizationTheorem}, and existence of such seed data in Theorem \ref{ExistenceProjection}.
Finally we show in \S\ref{Section-WeakExistence} that our methods also give rise to weakly asymptotically hyperbolic solutions that need not have sufficient boundary regularity for the shear-free condition to make sense.

\subsection*{Acknowledgement}
We thank Vincent Moncrief for helpful conversations.
This work was partially supported by NSF grants PHYS-1306441 and DMS-63431.

\section{Preliminaries}
\label{Section-Preliminaries}

\subsection{Einstein--Maxwell--fluid system}
\label{sect:EMF}

We study the  Einstein--Maxwell--perfect fluid system (in $3+1$ dimensions), 
in which the stress-energy tensor has the form
$
\mathbf T=\mathbf T^\text{Maxwell}+ \mathbf T^\text{fluid},
$
where $\mathbf T^\text{Maxwell}$ is determined  by the Maxwell $2$-form $\mathbf F$
via
\begin{equation}
\label{TMax}
\mathbf T^\text{Maxwell}_{\alpha\beta} = \mathbf F_{\alpha\mu}\mathbf F_{\beta\nu}\mathbf g^{\mu\nu}
-\frac14\left(\mathbf F_{\mu\nu}\mathbf F^{\mu\nu} \right)\mathbf g_{\alpha\beta},
\end{equation}
and $\mathbf T^\text{fluid}_{\alpha\beta} = p \mathbf g_{\alpha\beta} + (\mu + p) \mathbf u_\alpha \mathbf u_\beta$, where $\mathbf u$ is the unit timelike flow covector, $\mu\geq0$ is the energy density, and $p\geq 0$ is the pressure density. 
An \Defn{Einstein--Maxwell--perfect fluid solution} is a 
Lorentzian spacetime $(\mathbf M, \mathbf g)$, along with matter fields, 
satisfying 
\begin{equation}
\label{EinsteinMaxwellFluid}
\Ric[\mathbf g] - \frac{1}{2} \R[\mathbf g] \mathbf g = \mathbf T,
\end{equation}
together with Maxwell's equations
\begin{equation}
\label{MaxwellEquations}
\D \mathbf F =0
\quad \text{ and }\quad
\D(\ast \mathbf F) =0,
\end{equation}
and Euler's equations, which can be deduced from the conservation law
\begin{equation}
\label{ConsLaw}
\Div_{\mathbf g}\mathbf T^\text{fluid}=0.
\end{equation}
We refer the reader to \cite{ChoquetBruhat-GRBook} for details regarding these matter models.

In order to study the initial value problem for the system \eqref{EinsteinMaxwellFluid}-\eqref{TMax}, we focus on spacetime manifolds of the form  $\mathbf M = \mathbb R \times M$ for some smooth $3$-manifold $M$. Let $t$ be the coordinate on $\mathbb R$, so the slices $M_t=\{t\}\times M$ form a foliation of $\mathbf M$ whose leaves are diffeomorphic to $M$.
We assume that the metric induced on each leaf is Riemannian. 

Let $g$ denote the induced metric on each leaf $M_t$, and let
$K$ be the second fundamental form and $\tau = \tr_gK$ be the corresponding mean curvature 
of each leaf, all considered as time-dependent tensors on $M$. 
Note our sign convention for $K$: If $\mathbf D$ is the connection associated to $\mathbf g$, then we have $K(A,B) = \mathbf g(\mathbf{D}_AB,n)$.
The evolution vector $\partial_t$ need not be  orthogonal to the leaves of the foliation, but rather can be written as $\partial_t = Nn + X$, 
where $n$ is the future-pointing unit timelike normal to $ M_t\subset \mathbf M$ (i.e., ~$\mathbf g(n,n)=-1$, 
$\mathbf g(\partial_t, n)<0$
and $\mathbf g(n,W) =0$ if $W$ is tangent to $M_t$), $N>0$ 
is the {\it lapse function}, and the {\it shift vector} $X$ is tangent to each leaf.

We express the Einstein--Maxwell--fluid equations with respect to this product structure as follows.
Let $\iota_t\colon M \hookrightarrow \mathbf M$ be the inclusion taking $M$ to $M_t\subset \mathbf M$. 
We decompose the matter fields with respect to the foliation, setting
\begin{equation}
\begin{gathered}
E = \iota_t^*\left(\mathbf F (n, \cdot)\right)^\sharp,
\quad
B = \iota_t^*\left(\ast \mathbf F (n,\cdot)\right)^\sharp,
\\
\xi = \iota_t^*\left(\mathbf T^\text{fluid}(n,n)\right),
\quad
J = \iota_t^*\left(\mathbf T^\text{fluid}(n,\cdot)\right)^\sharp,
\end{gathered}
\end{equation}
where $\sharp$ is with respect to the time-dependent metric $g$ on $M$.

Expressed with respect to the foliation given by $t$, the Einstein field equations consist of the evolution equations for $g$ and $K$
\begin{align}
\label{Evolve-g}
\partial_t g &= -2NK + \mathcal L_Xg
\\
\label{Evolve-K}
\partial_tK&= N\Ric[g] +\mathcal L_XK + N\tau K -2N K\ast K - \Hess_gN 
-2N S,
\end{align}
together with the constraint equations
\begin{gather}
R[g] - |K|^2_g + \tau^2 = |E|^2_g + |B|^2_g + 2\xi
\label{HamiltonianConstraint}
\\
\Div_gK - \D\tau = \left(E\times_g\!B+ J\right)^\flat,  
\label{MomentumConstraint}
\end{gather}
which hold on each leaf of the foliation.
Here we define $K\ast K$ by $K\ast K(A,B) = \tr_g K(A,\cdot)\otimes K(B,\cdot)$, and the source term $S$ in \eqref{Evolve-K} is the restriction of $\mathbf T$ to $TM\times TM$.

The Maxwell equations \eqref{MaxwellEquations} give rise to evolution equations for $E$ and $B$, as well as to the Maxwell constraints
\begin{equation}
\label{MaxwellConstraints}
\Div_gE =0,
\quad\text{ and }\quad
\Div_gB=0;
\end{equation}
the evolution equations for the fluid are simply Euler's equations, which are not subject to constraints. 
Data for the initial value problem thus consists of $g$, $K$, $E$, $B$, $\xi$, and $J$, all defined on $M$ and satisfying \eqref{HamiltonianConstraint}, \eqref{MomentumConstraint}, and \eqref{MaxwellConstraints}.

\subsection{Compactification and the shear-free condition}
The discussion in Section \ref{sect:EMF} of the decomposition of the metric and other fields with respect to a spacelike foliation, and the resulting decomposition of the field equations into constraints and evolution equations, hold for any spacelike foliation of any spacetime, regardless of their asymptotic properties. 
We now presume that the spacetime is asymptotically flat and investigate the behavior of the field equations under conformal compactification of the geometry.

Writing $\mathbf{M}$ as the interior of $\mathbf{\bar M}$ and setting $\mathbf g = \Omega^{-2}\mathbf{\bar g}$, we find that if the leaves of a foliation of $\mathbf{M}$ approach  $\mathscr I^+$ transversely, then there is a corresponding foliation of $\mathbf{\bar M}$ such that each leaf $M$ of the former foliation is the interior of a compact manifold $\bar M$ (with boundary $\partial M$) which is a leaf of the latter foliation. Moreover, the metric  $g$ induced on $M$ can be expressed as 
\begin{equation}
\label{ConvertMetric}
g = \Omega^{-2}\bar g,
\end{equation}
with $\bar g$ the restriction of $\mathbf{\bar g}$ to $\bar M$.
The assumption that the foliation approaches $\mathscr I^+$ transversely implies that $\Omega$ plays the role of a defining function for $M \subset \bar M$. 
We assume that 
$|\D\Omega|^2_{\bar g} =1$ along $\partial M$, and that consequently $(M,g)$ is asymptotically hyperbolic.

Analogous to the discussion in \S \ref{sect:EMF}, one can decompose $\partial_t$ with respect to $\mathbf{\bar g}$, yielding $\partial_t = \bar N \bar n + \bar X$; one easily sees that $\bar n = \Omega^{-1}n$, and hence $\bar N = \Omega N$ and $\bar X = X$.
Let $\bar K$ be the second fundamental form of $(\bar M,\bar g)\subset (\mathbf{\bar M},\mathbf{\bar g})$ and $\bar \tau = \tr_{\bar g}\bar K$ be the associated mean curvature.
Letting $\Sigma$ be the traceless part of $K$, and $\bar\Sigma$ the traceless part of $\bar K$, we have
\begin{equation}
\label{DecomposeK}
K = \Sigma + \frac{\tau}{3} g
= \Omega^{-1}\bar\Sigma + \frac{\tau}{3}\Omega^{-2} \bar g
\end{equation}
and
\begin{equation}
\label{RelateHtoTau}
\frac{\bar\tau}{3}
= \Omega^{-1}\frac{\tau}{3} - \bar n(\log\Omega).
\end{equation}

Substituting \eqref{ConvertMetric} and \eqref{DecomposeK} into the constraint equations \eqref{HamiltonianConstraint} and \eqref{MomentumConstraint} we obtain the following:
\begin{equation}
\label{FirstLichOmega}
4\Omega \Delta_{\bar g}\Omega 
+ (R[\bar g] - |\bar\Sigma|^2_{\bar g})\Omega^2 
+ 6\left(\frac{\tau^2}{9} - |\D\Omega|^2_{\bar g} \right)
= 
| E|^2_{ g} + | B|^2_{ g} + 2\xi,
\end{equation}
\begin{equation}
\label{FirstDivSigmaBar}
\Omega \Div_{\bar g}\bar\Sigma 
- 2\bar\Sigma(\grad_{\bar g}\Omega, \cdot)
-\frac23\Omega^{-1}\D\tau
=\left( E\times_{ g}\! B +  J \right)^\flat;
\end{equation}
here our sign convention for the Laplacian is $\Delta_{\bar g}  = \Div_{\bar g}\circ\grad_{\bar g} = \tr_{\bar g}\Hess_{\bar g}$ and the symbol $\flat$  on the right side of \eqref{FirstDivSigmaBar} is with respect to $g$.

We assume that $|E|^2_g$, $|B|^2_g$, $|J|_g$, and $\xi$ fall off like $\mathcal O(\Omega^2)$.
Thus if $\bar g\in C^2(\bar M)$, if $\bar\Sigma\in C^1(\bar M)$, and if $\Omega\in C^2(\bar M)$ satisfies 
\begin{equation}
\label{LichOmegaBC}
\Omega\Big|_{\partial M} =0,
\quad\text{ and }\quad
|\D\Omega|_{\bar g}^2\Big|_{\partial M} =1,
\end{equation}
then 
\eqref{FirstLichOmega} is consistent on $\bar M$ if and only if we require that $\tau^2=9$ on $\partial M$.
Consequently, imposing the constant-mean-curvature (CMC) condition requires $\tau^2 = 9$.

It follows from the condition that $\partial_t$ be tangential to $\mathscr I^+$ and future-directed that one has $X(\Omega) = \bar g(X,\grad_{\bar g}\Omega) >0$ on $\partial M$. 
From \eqref{RelateHtoTau} we have $\tau = \Omega \bar\tau + 3\bar N^{-1}( \partial_t\Omega - X(\Omega))$; thus on $\partial M$ we have $\tau = -3\bar N^{-1}X(\Omega)<0$.
Consequently the CMC condition becomes $\tau = -3$.

We also note that in the CMC setting \eqref{FirstDivSigmaBar} implies that $\bar\Sigma(\grad_{\bar g}\Omega, \cdot) =0$ at $\partial M$.

The constraint system for hyperboloidal data has often been viewed as consisting of \eqref{HamiltonianConstraint}-\eqref{MomentumConstraint} 
together with \eqref{MaxwellConstraints}, the asymptotically hyperbolic condition on $g$, the boundary conditions \eqref{LichOmegaBC} along $\partial M$, and suitable fall-off conditions on $\bar \Sigma$  and on the non-gravitational fields. However, solutions of this system do not generally evolve into asymptotically flat spacetimes with a well-defined $\mathscr I^+$ unless a further boundary condition is imposed. 
To motivate this condition, we write out the evolution equation for $\bar \Sigma$ (calculated using \eqref{Evolve-g} and  \eqref{Evolve-K}):
\begin{multline}
\label{ShearfreeMotivate}
\partial_t\bar\Sigma =\mathcal L_{\bar X}\bar\Sigma
 -2\bar N \,\bar\Sigma\ast\bar\Sigma 
-\bar N\frac{H}{3}\bar\Sigma
+\bar N\left(\Ric[\bar g]-\frac 13R[\bar g]\,\bar g\right)  
\\
-\left(\Hess_{\bar g} \bar N -\frac13(\Delta_{\bar g}\bar N)\bar g\right)\\
\quad + 2\Omega^{-1}\left(\Hess_{\bar g}\Omega -\frac13(\Delta_{\bar g}\Omega)\bar g -\bar\Sigma \right) 
-\bar N S.
\end{multline}

It follows that if the evolution is to be nonsingular on $\bar M$, and if the resulting spacetime is to satisfy 
$\mathbf{\bar g}\in C^{1,1}(\mathbf{\bar M})$, then the initial data must satisfy 
\begin{equation}
\label{FirstShearFree}
\left.\bar\Sigma\right|_{\partial M} = \left[\Hess_{\bar g}\Omega -\frac13(\Delta_{\bar g}\Omega)\bar g\right]_{\partial M}.
\end{equation}
This boundary  condition \eqref{FirstShearFree} is known in the literature as the \Defn{shear-free condition}. It is worth emphasizing that the shear-free condition is {\it not} affected by the presence of matter fields.

To understand the geometric content of the shear-free condition, we start by noting that it follows from the  constraint \eqref{FirstLichOmega}  that  
$|\D\Omega|^2_{\bar g} = 1 + \frac23 \Omega \Delta_{\bar g}\Omega + \mathcal O(\Omega^2)$ as $\Omega \to 0$.
Taking the derivative of both sides of this equation and combining with the identity  
$\Hess_{\bar g}\Omega(\grad_{\bar g}\Omega, \cdot) = \frac12\D (|\D\Omega|^2_{\bar g})$
yields
$
\Hess_{\bar g}\Omega(\grad_{\bar g}\Omega, \cdot) - \frac13(\Delta_{\bar g}\Omega)\D\Omega =0
$
along the boundary. Evaluating the momentum constraint \eqref{FirstDivSigmaBar} at $\partial M$, we have $\bar\Sigma(\grad_{\bar g}\Omega, \cdot) =0$. Thus
\begin{equation}
\label{ShearFree-Transverse}
\bar\Sigma(\grad_{\bar g}\Omega, \cdot) = \Hess_{\bar g}\Omega(\grad_{\bar g}\Omega, \cdot) - \frac13(\Delta_{\bar g}\Omega)\bar g(\grad_{\bar g}\Omega, \cdot)
\end{equation}
holds trivially along $\partial M$. This tells us that the only nontrivial content of \eqref{FirstShearFree} lies in its restriction to $T(\partial M)\times T(\partial M)$.

Let $\chi_{\bar g}$ be the second fundamental form of $\partial M \subset \bar M$ with respect to the inward normal $\grad_{\bar g}\Omega$, and let $\widehat\chi_{\bar g}$ be its trace-free part.
The shear-free condition \eqref{FirstShearFree} then reduces simply to
\begin{equation}
\label{ShearFree-Tangential}
\left.\bar\Sigma\right|_{T(\partial M)\times T(\partial M)} + \widehat\chi_{\bar g}=0.
\end{equation}

The null hypersurface $\mathscr I^+$ is foliated by null $\overline{\mathbf g}$-geodesics, 
generated by a certain null vector field $k$ (see \cite[Chap.~11]{Wald}).
The \Defn{shear} of $\mathscr I^+$ is the trace-free part of the Lie derivative $\mathcal L_k \big({\left.\overline {\mathbf g}\right|}_{T(\partial M)}\big)$;
it measures the extent to which the flow of $k$
deforms spheres of the geometry of $\partial M$ towards ellipsoids. 
Because $k$ can be expressed as the sum of the unit future-timelike
and unit inward-spacelike vectors, 
it follows that the shear vanishes if and only if 
$\left.\bar\Sigma\right|_{T(\partial M)\times T(\partial M)}+\widehat\chi_{\bar g}$ does.

The shear-free condition plays a central role in \cite{AnderssonChrusciel-Obstructions}, where the non-vanishing of the shear of initial data is identified as an obstruction to the existence of polyhomogeneous $\mathscr I$ for any spacetime which might arise from such data; it also plays an important role in the formal computations of \cite{MoncriefRinne-Regularity}.
These results suggest that imposing the shear-free condition is necessary, and may be sufficient, for constructing spacetimes admitting $\mathscr I$ with at least some minimal regularity.
Unfortunately, as noted in the introduction, many results concerning the existence of hyperboloidal initial data do not explicitly address the existence of data satisfying the shear-free condition.
One of our  purposes here is to clarify the situation regarding asymptotically hyperbolic data which does satisfy the shear-free condition.
A first step in providing such clarity is the formulation of a careful definition of such data.

\section{Hyperboloidal initial data}
\label{Section-DefineData}
Observe that the conformal compactification of an asymptotically hyperbolic manifold $(M,g)$ is not unique:
If $g = \Omega^{-2}\bar g$ for some defining function $\Omega$, and if $\omega \in C^1(\bar M)$ is positive, then 
\begin{equation}
|\D(\omega\Omega)|^2_{\omega^2\bar g}\Big|_{\partial M}
=
|\D\Omega|^2_{\bar g}\Big|_{\partial M}.
\end{equation}  
In fact, if $\Omega$, $\bar g$, $\bar\Sigma$, and a collection of matter fields $\Phi = ( E, B, J, \xi)$ give rise, via \eqref{ConvertMetric} and \eqref{DecomposeK}, to a solution to the constraint equations \eqref{HamiltonianConstraint}--\eqref{MaxwellConstraints}, then  $\omega \Omega$, $\omega^2\bar g$, $\omega\bar\Sigma$ give rise to the same solution.

In view of these observations, we fix once and for all a smooth defining function $\rho$, which we use in place of $\Omega$ for the purposes of conformal compactification, and for indicating the regularity of various fields along $\partial M$.

We furthermore fix a smooth background metric $\bar h$ on $\bar M$ such that $|\D\rho|_{\bar h} = 1$ along $\partial M$, and let $h = \rho^{-2}\bar h$ be the corresponding asymptotically hyperbolic metric on $M$.
We denote by $\bar\nabla$ and by $\nabla$ the Levi-Civita connections associated to $\bar h$ and to $h$, respectively.
The connection associated to any other metric $g$ is denoted by $\Nabla{g}$.

The metrics $\bar h$ and $ h$, together with $\rho$, are used to define various regularity classes of tensor fields, including the weighted H\"older spaces $C^{k,\alpha}_\delta(M) = \rho^{\delta}C^{k,\alpha}(M)$, as well as $C^k_\phg(\bar M)$, those polyhomogeneous tensor fields on $M$ with extensions to  $\bar M$ of class $C^k$. 
Our convention is that norms for function spaces on $\bar M$ are computed with respect to $\bar h$, and norms for function spaces on $M$ are computed with respect to $h$.
The reader is referred to \cite{WAH-Preliminary} and \cite{Lee-FredholmOperators} for details regarding these spaces.

Note that if $u$ is a tensor field of covariant rank $p$ and contravariant rank $q$, then 
\begin{equation}
\label{ConvertToBar}
|u|_h = \rho^{r}|u|_{\bar h},
\end{equation}
where $r = p-q$ is the \Defn{weight} of the tensor $u$.
Consequently, if $u\in C^{0}_{r}(M)$, then $u\in L^\infty(\bar M)$.
In \cite{WAH-Preliminary}, we introduce for $0\leq m \leq k$ and $\alpha\in [0,1)$ the spaces $\mathscr C^{k,\alpha;m}(M)$ of tensor fields $u$ such that if $r$ is the weight of $u$ then
\begin{equation}
\label{DefineTripleNorm}
\gNorm{u}_{k,\alpha;m}
:= \sum_{l=0}^m \| \bar\nabla{}^l u\|_{C^{k-l,\alpha}_{r+l}(M)}
\end{equation}
is finite.
These spaces are intermediate between $C^{k,\alpha}_r(M)$ and $C^{k,\alpha}(\bar M)$.

In view of \eqref{ConvertToBar}, if $\gNorm{u}_{k,\alpha;m}$ is finite then $|u|_{\bar h},|\bar\nabla u|_{\bar h}, \dots, |\bar\nabla{}^m u|_{\bar h}$ are bounded.
Thus if $u\in \mathscr C^{k,\alpha;m}(M)$, then $u$ extends to a tensor field on $\bar M$ of class $C^{m-1,1}(\bar M)$.

Throughout this paper we adopt the convention that if $g$ is a metric on $M$, then $\bar g$ is the metric given by $\bar g = \rho^2 g$;
in the situations considered below, it is always the case that $\bar g$ extends continuously to a metric $\bar M$, which we call the \Defn{conformal compactification of $g$}.
We define a metric $g$ on $M$ to be \Defn{weakly $C^{k,\alpha}$ asymptotically hyperbolic} if $\bar g \in \mathscr C^{k,\alpha;1}(M)$ is non-degenerate on $\bar M$ and if $|\D\rho|_{\bar g} = 1$ along $\partial M$.
If $g$ is weakly $C^{1,0}$ asymptotically hyperbolic, then $\bar g$ extends to a Lipschitz continuous metric on $\bar M$.
We denote the collection of weakly $C^{k,\alpha}$ asymptotically hyperbolic metrics on $M$ by $\WAH^{k,\alpha;1}$.
In \cite{WAH-Preliminary}, we establish Fredholm results for elliptic operators defined using metrics in $\WAH^{k,\alpha;1}$; see Theorem \ref{Fred} below.

Metrics in $\WAH^{k,\alpha;1}$ have sufficient regularity for the elliptic theory needed to construct solutions to the constraint equations (see Theorem \ref{WeakExistenceTheorem} below), but this regularity is not sufficient for analyzing the shear-free condition.
Thus we seek a space of metrics for which the shear-free condition may be defined.
We furthermore require that the space have a topology that is sufficiently strong  for the shear-free condition to be a closed condition, yet  weak enough to allow the conformal structures at infinity of shear-free initial data sets to vary continuously.

The density result of \cite{SFDensity-Prelim} shows in the vacuum setting that the topology induced by the unweighted H\"older spaces is too weak for the shear-free condition to be closed.
In order to gain some intuition for this result, let us suppose that $g\in C^{k,\alpha}(M)$ for some $k\geq 2$ and $\alpha\in [0,1)$, and thus that $\bar g = \rho^2 g \in C^{k,\alpha}_2(M)$ and $(\bar g)^{-1}\in C^{k,\alpha}_{-2}(M)$.
The trace-free Hessian appearing in \eqref{FirstShearFree}, with $\Omega$ replaced by $\rho$, is given schematically by contractions of
\begin{equation}
\label{HessianScheme}
 (\bar g)^{-1} \otimes \bar g\otimes  \bar\nabla\D\rho
\quad\text{ and }\quad
 (\bar g)^{-1} \otimes  (\bar g)^{-1} \otimes \bar g\otimes  \bar\nabla \bar g\otimes  \D\rho.
\end{equation}
As $\D\rho\in C^{k,\alpha}_1(M)$ and $\bar\nabla \D\rho\in C^{k,\alpha}_2(M)$ for all $k$ and $\alpha$, we see that $C^{k,\alpha}$ control of $g$ yields an estimate for the trace-free Hessian of $\rho$ in $C^{k-1,\alpha}_1(M)$.
However, as  a  covariant $2$-tensor field, an estimate of the trace-free Hessian of $\rho$ in $L^\infty(\bar M)$ requires an estimate in $C^{k-1,\alpha}_2(M)$; cf.~\eqref{ConvertToBar}.
Furthermore, in order for the trace-free Hessian to be defined pointwise on $\partial M$, we need even more regularity.

Based on these observations, one might be tempted to simply increase the weight, and measure variations of $g$ in the $C^{k,\alpha}_1$ norm.
However, if a sequence  of asymptotically hyperbolic metrics $g_i$ is to converge to $g$ in $C^{k,\alpha}_1(M)$ it is necessary that each $\bar g_i = \rho^2g_i$ agree with $\rho^2 g$ along $\partial M$.
Thus this weighted topology is too strong to allow the conformal structure at infinity to vary continuously.

We note, however, that terms of the first type in \eqref{HessianScheme} can be estimated in $C^{k-1,\alpha}_2(M)$ if $g\in C^{k,\alpha}(M)$; it is only terms of the second type, which involve derivatives of the metric, that lead to a loss of weight.
Such terms, however, are bounded if $g\in \WAH^{k,\alpha;1}$.
We can furthermore ensure that the traceless Hessian extends to a Lipschitz continuous tensor on $\bar M$ if we require $\bar g \in \mathscr C^{k,\alpha;2}(M)$.
Thus we define $\WAH^{k,\alpha;2}$ to be those metrics $g\in \WAH^{k,\alpha;1}$ with $\bar g \in \mathscr C^{k,\alpha;2}(M)$.
The preceding discussion is summarized by the following.
\begin{lemma}
Suppose $k\geq 2$ and $\alpha \in [0,1)$.
\begin{enumerate}
\item If $g\in \WAH^{k,\alpha;1}$ then 
$
\Hess_{\bar g}\rho - \frac13 (\Delta_{\bar g}\rho)\bar g \in L^\infty(\bar M).
$
\item If $g\in \WAH^{k,\alpha;2}$ then 
$
\Hess_{\bar g}\rho - \frac13 (\Delta_{\bar g}\rho)\bar g \in C^{0,1}(\bar M).
$
\end{enumerate}
\end{lemma}

We now define hyperboloidal initial data sets as follows.
\begin{defn}
\label{DefineCMCSFAH}
Let $k\geq 2$ and $\alpha \in (0,1)$.
A set of fields $(g,K,\Phi)$ constitutes a \Defn{constant-mean-curvature shear-free hyperboloidal initial data set of class $C^{k,\alpha}$} on $M$ if 
\begin{enumerate}
\item  \label{DD-g}
$g\in \WAH^{k,\alpha;2}$;  

\item \label{DD-K}
$K$ is a symmetric covariant $2$-tensor field of the form $K = \Sigma - g$ with $\Sigma$  traceless and $\bar\Sigma=\rho\Sigma\in \mathscr C^{k-1,\alpha;1}(M)$;

\item
\label{DD-Phi}
 $\Phi = (E,B,J,\xi)$ is a collection of matter fields satisying
 \begin{equation}
\label{MatterRegularity}
\Phi\in 
C^{k-2,\alpha}_1(M) \times C^{k-2,\alpha}_1(M) \times C^{k-2,\alpha}_2(M) \times C^{k-2,\alpha}_2(M)
\end{equation}
 and such that $\xi\geq 0$;

\item 
\label{DD-Constraints}
$g$, $K$, and $\Phi$ satisfy the constraint equations \eqref{HamiltonianConstraint}--\eqref{MaxwellConstraints}; and

\item 
\label{DD-BC}
the shear-free condition \eqref{FirstShearFree} is satisfied with $\Omega = \rho$.

\end{enumerate}
We denote  the collection of such initial data sets by $\mathscr D^{k,\alpha}$, and distinguish the following subsets:

\begin{itemize}
\item We define the set $\mathscr D^\infty$ of \Defn{smooth initial data sets } to be those data sets in $\mathscr D^{k,\alpha}$ for all $k\geq 2$ and all $\alpha$.

\item The set $\mathscr D_\phg$ of  \Defn{polyhomogeneous initial data sets} are those for which
\begin{equation}
\label{PhysicalRegularity}
(\rho^2 g, \rho\Sigma,\rho\odot\Phi)\in C^2_\phg(\bar M)\times C^1_\phg(\bar M)\times C^0_\phg(\bar M),
\end{equation}
where we define $\rho \odot\Phi $ by
\begin{equation}
\rho\odot ( E,  B, J, \xi) 
:= (\rho^{-3} E, \rho^{-3} B,\rho^{-5} J, \rho^{-4}\xi).
\end{equation}

\end{itemize}

\end{defn}

Below we discuss the parametrization and construction of data sets  via the conformal method. Before doing this, we introduce an alternate characterization of the shear-free condition \eqref{FirstShearFree} that is particularly well suited for use with the conformal method.

\section{The tensor $\B_{\bar g}(\rho)$}
\label{Section-B}

While  the condition \eqref{FirstShearFree} precisely characterizes the shear-free condition, in order to incorporate the condition into the construction of hyperboloidal data,  it is useful to have an alternative characterization in terms of quantities that are conformally invariant.
We accomplish this making use of the tensor field $\B_{\bar g}(\rho)$, defined in \cite{WAH-Preliminary} by
\begin{equation}
\label{DefineB}
\B_{\bar g}(\rho)
:=|\D\rho|_{\bar g}^6\,\mathcal D_{\bar g}(|\D\rho|^{-2}_{\bar g} \grad_{\bar g}\rho)
+ A_{\bar g}(\rho) \left( \D\rho \otimes \D\rho - \frac{1}{3}|\D\rho|^2_{\bar g} \bar g \right),
\end{equation}
where
\begin{equation}
A_{\bar g}(\rho) 
:= \frac{1}{2} |\D\rho|_{\bar g} \Div_{\bar g}\left[ |\D\rho|_{\bar g}\grad_{\bar g}\rho\right]
\end{equation}
and $\mathcal D_{\bar g}$ is the \Defn{conformal Killing \(or Alhfors\) operator} taking a vector field $X$ to a symmetric and tracefree covariant $2$-tensor field by
\begin{equation}
\label{DefineAlhforsOperator}
\mathcal D_{\bar g}X = \frac12\mathcal L_{X}\bar g - \frac{1}{3}(\Div_{\bar g}X)\bar g.
\end{equation}
A vector field $X$ such that $\mathcal D_{\bar g}X =0$ is called a \Defn{conformal Killing vector field}.

The tensor field $\B_{\bar g}(\rho)$ is a conformally invariant version of the trace-free Hessian. 
We recall the following results from \cite{WAH-Preliminary}.

\begin{proposition}[Proposition 4.1 of \cite{WAH-Preliminary}]
\hfill
\label{B-BasicProperties}
\begin{enumerate}
\item $\B_{\bar g}(\rho)$ is symmetric and trace-free.

\item\label{B-TransverseProperty} $\B_{\bar g}(\rho)(\grad_{\bar g}\rho, \cdot)=0$.

\item\label{B-Scaling} $\B_{\bar g}(c \rho)=c^5\B_{\bar g}(\rho)$ for all constants $c$.

\item\label{B-ConformalScaling} If $\theta$ is a strictly positive $C^1$ function then 
$\B_{\theta^4\bar g}(\rho)=\theta^{-8}\B_{\bar g}(\rho)$
and $A_{\theta^4\bar g}(\rho) = \theta^{-8}A_{\bar g}(\rho)$.
\end{enumerate}
\end{proposition}

\begin{proposition}[Lemma 4.2 of \cite{WAH-Preliminary}]
\hfill
\label{B-in-WAH}
\begin{enumerate}
\item 
Let $g\in \WAH^{k,\alpha;1}$ for $k\geq 1$ and $\alpha\in [0,1)$, and let $\bar g = \rho^2 g$.
Then $\B_{\bar g}(\rho)\in C^{k-1,\alpha}_2(M)$.

\item
If $g\in \WAH^{k,\alpha;2}$ for  $k\geq 2$ and $\alpha\in [0,1)$, then $\bar\nabla \B_{\bar g}(\rho)\in C^{k-2,\alpha}_3(M)$ and thus $\Div_{\bar g}\B_{\bar g}(\rho) \in C^{k-2,\alpha}_1(M)$.
In particular, $\B_{\bar g}(\rho)\in \mathscr C^{k-1,\alpha;1}(M)$ and therefore $\B_{\bar g}(\rho)$ extends continuously to $\bar M$.
\end{enumerate}

\end{proposition}

\begin{proposition}[Proposition 4.3 of \cite{WAH-Preliminary}]
\label{B-DefiningFunctionProperties}
Suppose  $g\in \WAH^{k,\alpha;2}$ for some $k\geq 2$ and $\alpha\in (0,1)$, and $\R[g]+6\in C^{k-2,\alpha}_2(M)$.
Then $\B_{\bar g}(\rho)$ satisfies
\begin{equation}
\label{BisHessian}
\B_{\bar g}(\rho) - \left(\Hess_{\bar g}\rho - \frac13(\Delta_{\bar g}\rho)\bar g\right)\in C^{k-1,\alpha}_3(M).
\end{equation}
In particular, 
\begin{equation}
\left.\phantom{\frac13}\B_{\bar g}(\rho)\right|_{\partial M} 
=
 \left[\Hess_{\bar g}\rho - \frac13(\Delta_{\bar g}\rho)\bar g\right]_{\partial M}.
\end{equation}

\end{proposition}

Proposition \ref{B-DefiningFunctionProperties} implies that the shear-free condition can be expressed using $\B_{\bar g}(\rho)$.

\begin{corollary}
\label{BShearFree}
Suppose that $(g,K,\Phi)$ satisfy parts \eqref{DD-g}--\eqref{DD-Constraints} of Definition \ref{DefineCMCSFAH} with $k\geq 2$.
Then the following are equivalent:
\begin{enumerate}
\item 
\label{part:OriginalSF}
The shear-free condition of part \eqref{DD-BC} in Definition \ref{DefineCMCSFAH} is satisfied.

\item 
\label{part:PhysicalShearFree}
$
\bar\Sigma\big|_{\partial M} = \left.\B_{\bar g}(\rho)\right|_{\partial M}.
$

\item 
\label{part:C-ShearFree}
$
\bar\Sigma - \B_{\bar g}(\rho) \in C^{k-1,\alpha}_3(M).
$

\end{enumerate}

\end{corollary}

\begin{proof}
By virtue of Proposition \ref{B-DefiningFunctionProperties}, \eqref{part:C-ShearFree} $\implies$ \eqref{part:OriginalSF} $\implies$ \eqref{part:PhysicalShearFree}.
To see that \eqref{part:PhysicalShearFree} implies \eqref{part:C-ShearFree}, we first note that by Definition \ref{DefineCMCSFAH} and Proposition \ref{B-in-WAH} we have that $\bar\Sigma - \B_{\bar g}(\rho)$ is in $\mathscr C^{k-1,\alpha;1}(M)$ and vanishes as $\rho \to 0$.
Thus
Lemma 2.2(d) of \cite{WAH-Preliminary}
implies that $\bar\Sigma - \B_{\bar g}(\rho)\in C^{k-1,\alpha}_3(M)$.
\end{proof}

\section{The conformal method}
\label{Section-ConformalMethod}
The conformal method provides a means for constructing solutions to the constraint equations and, in the constant-mean-curvature case, for parametrizing the set of all solutions. See \cite{Isenberg-CMCclosed} for a comprehensive discussion of this application for the case of data on compact manifolds;
 \cite{Maxwell-ConformalMethodsSame} contains a brief history and shows equivalence to other data-construction methods; see also the discussion in \cite{ChoquetBruhat-GRBook}.

\subsection{Parametrization}
The idea behind using the conformal method as a means of parametrizing families (such as  $\mathscr D^{k,\alpha}$, $\mathscr D^\infty$, $\mathscr D_\phg$) of initial data sets is to find a space   of {seed data sets} such that each seed data set can be conformally changed in order to obtain an initial data set, with the conformal factor obtained by solving an elliptic partial differential equation. 
If each seed data set is conformally related to precisely one initial data set, then one may parametrize the family of initial data sets by a quotient of the collection of seed data sets.

For $k\geq 2$ and $\alpha\in (0,1)$ we define the collection of \Defn{$C^{k,\alpha}$ seed data sets} appropriate for parametrizing hyperboloidal data to be those tuples $(\lambda,\sigma,\Psi)$ where
\begin{enumerate}

\item $\lambda\in \WAH^{k,\alpha;2}$;

\item $\Psi = (\mathcal E,\mathcal B,j,\zeta)$, where $\mathcal E$ and $\mathcal B$ are vector fields that are divergence-free with respect to $\lambda$, $j$ is a vector field, and $\zeta$ is a non-negative function, all having the regularity given by \eqref{MatterRegularity}; and

\item $\sigma$ is a symmetric covariant $2$-tensor field that is trace-free with respect to $\lambda$, satisfying $\rho\sigma\in \mathscr C^{k-1,\alpha;1}(M)$ and 
\begin{gather}
\label{NewMomentumForSigma}
(\Div_\lambda\sigma)^\sharp =j + \mathcal E \times_\lambda\mathcal B,
\\
\label{sigmaBC}
\left.\rho\sigma \right|_{\partial M} = \left.\B_{\rho^2\lambda}(\rho)\right|_{\partial M}.
\end{gather}
\end{enumerate}

We denote by $\mathscr S^{k,\alpha}$ the collection of such seed data sets, and distinguish the following subsets:

\begin{itemize}

\item The collection $\mathscr S^\infty$ of \Defn{smooth seed data sets} is defined to be those tuples which are in $\mathscr S^{k,\alpha}$ for all $k\geq 2$ and all $\alpha$.

\item The collection $\mathscr S_{\phg}$ of \Defn{polyhomogeneous seed data sets}  is defined to be those tuples for which $\rho^2\lambda\in C^2_\phg(\bar M)$, $\rho\sigma \in C^1_\phg(\bar M)$,  and $\rho\odot\Psi \in C^0_\phg(\bar M)$.

\end{itemize}

Given a seed data set $(\lambda, \sigma, \Psi)$ we seek to construct a CMC shear-free hyperboloidal initial data set $(g, K, \Phi)$ of the form
\begin{equation}
\label{ConfDecomp}
g = \phi^4 \lambda, 
\quad
K= \phi^{-2} \sigma -g,
\quad
\Phi = \phi^2\odot\Psi,
\end{equation}
where $\phi$ is an unknown positive function. 
We note that for any symmetric, trace-free covariant $2$-tensor field $T$ and positive function $\phi$ we have
\begin{equation*}
\Div_{\phi^4\lambda}(\phi^{-2}T) = \phi^{-6}\Div_\lambda T,
\end{equation*}
while the scalar curvature changes under conformal transformations as
\begin{equation}
\label{ConformalChangeScalarCurvature}
\R[\phi^4\lambda] = \phi^{-4}\R[\lambda] - 8\phi^{-5}\Delta_\lambda\phi.
\end{equation}
Thus inserting the fields $(g,K,\Phi)$, as given by \eqref{ConfDecomp}, into \eqref{HamiltonianConstraint}--\eqref{MomentumConstraint}, we see that the constraint equations are satisfied  so long as $\phi$ satisfies the \Defn{Lichnerowicz equation}
\begin{equation}
\label{LichPhi}
\Delta_{\lambda}\phi - \frac18 \R[\lambda]\phi 
+ \frac18 |\sigma|^2_{\lambda} \phi^{-7} 
+\frac18 \left(|\mathcal E|^2_{ \lambda} +|\mathcal B|^2_{\lambda}  + 2\zeta \right)\phi^{-3}
- \frac34 \phi^5=0.
\end{equation}
In order for $(g,K,\Phi)$ to satisfy the (asymptotically hyperbolic and shear-free) boundary conditions in Definition \ref{DefineCMCSFAH}, we require $\phi =1$ along $\partial M$.

One may readily verify that the conditions defining seed data are invariant under the transformation
\begin{equation}
\label{ConformalCovarianceMapping}
( \lambda, \sigma,\Psi)\mapsto (\theta^4 \lambda, \theta^{-2}\sigma, \theta^2\odot\Psi)
\end{equation}
for any suitably regular, positive function $\theta$ with $\theta=1$ along $\partial M$.
Furthermore, both the ansatz \eqref{ConfDecomp} and  the Lichnerowicz equation \eqref{LichPhi} are invariant under \eqref{ConformalCovarianceMapping} in following sense.

\begin{lemma}
\label{lemma:LichInvariance}
Suppose $(\lambda, \sigma, \Psi)\in \mathscr S^{k,\alpha}$ and $\theta$ is a positive function  with $\theta-1\in C^{k,\alpha}_1(M)$; let $( \tilde\lambda, \tilde\sigma,\tilde\Psi)= (\theta^4 \lambda, \theta^{-2}\sigma, \theta^2\odot\Psi)$.
Then $\phi$ satisfies the Lichnerowicz equation \eqref{LichPhi} if and only if $\tilde\phi = \theta^{-1}\phi$ satisfies the Lichnerowicz equation corresponding to $( \tilde\lambda, \tilde\sigma,\tilde\Psi)$.
Furthermore, both $(\lambda, \sigma, \Psi)$ and $( \tilde\lambda, \tilde\sigma,\tilde\Psi)$ give rise to the same initial data set $(g,K,\Phi)$. 
\end{lemma}
\begin{proof}
Direct computation using \eqref{ConformalChangeScalarCurvature}.
\end{proof}

In \S\ref{Section-PandE} we prove, using the results of \S\S\ref{PDEResults}--\ref{PDE-Continuity}, that the various classes of initial data in Definition \ref{DefineCMCSFAH} are homeomorphic to quotients of the corresponding classes of seed data by the equivalence relation arising from \eqref{ConformalCovarianceMapping}.

\subsection{Existence}
While the parametrization theorems are useful for understanding the structure of the spaces of initial data sets, the conditions defining the seed data sets are non-trivial.
In \S\ref{Section-PandE}, we also address the issue of existence of seed data sets by showing that all such data sets can be obtained by choosing certain arbitrary fields, which we refer to as ``free data,'' and projecting to a seed data set as follows.

For $k\geq 3$ and $\alpha \in (0,1)$, we define a \Defn{$C^{k,\alpha}$ free data set} to be a tuple $(\lambda, \nu, \Upsilon)$, where 
\begin{enumerate}
\item $\lambda \in \WAH^{k,\alpha;2}$;
\item $\nu\in C^{k-1,\alpha}_2(M)$ is a symmetric traceless covariant $2$-tensor field;
\item $\Upsilon = (e,b,j,\zeta)$ consists of vector fields $e$, $b$, $j$, and a non-negative function $\zeta$, and has regularity as in \eqref{MatterRegularity}.
\end{enumerate}
We denote by $\mathscr F^{k,\alpha}$ the collection of all such data sets, and distinguish the following subsets:
\begin{itemize}
\item The collection $\mathscr F^{\infty}$ of \Defn{smooth free data sets} consists of those tuples in $\mathscr F^{k,\alpha}$ for all $k\geq 3$ and all $\alpha$.

\item The collection $\mathscr F_\phg$ of \Defn{polyhomogeneous free data sets} consists of those tuples for which
\begin{equation*}
(\rho^2\lambda, \nu, \rho\odot\Upsilon)\in C^2_\phg(\bar M)\times C^0_\phg(\bar M) \times C^0_\phg(\bar M).
\end{equation*}

\end{itemize}

\begin{remark}
The assumption $k\geq 3$ in the definition of $\mathscr F^{k,\alpha}$ is needed in order to solve \eqref{FirstVL} below; see Lemma \ref{TrivialCKF} and Remark \ref{WhyK3}.
\end{remark}

Given a free data set $( \lambda, \nu,\Upsilon)$, one may obtain a seed data set $(\lambda, \sigma, \Psi)$ by constructing $\Psi$ and $\sigma$ according to the following procedure.
Define the matter fields $\Psi = (\mathcal E, \mathcal B, j, \zeta)$ by
\begin{equation}
\label{ConstructEB}
\mathcal E =  e - \grad_\lambda u
\quad\text{ and }\quad
\mathcal B =  b - \grad_\lambda v,
\end{equation}
which are divergence-free provided the functions $u$ and $v$ satisfy
\begin{equation}
\label{LaplaceUV}
\Delta_\lambda u = \Div_\lambda  e
\quad \text{ and }\quad
\Delta_\lambda v = \Div_\lambda b.
\end{equation}

In order to construct a tensor field $\sigma$ satisfying \eqref{NewMomentumForSigma}, we recall the conformal Killing operator $\mathcal D_{\lambda}$ defined by \eqref{DefineAlhforsOperator}.
The formal $L^2$ adjoint $\mathcal D_\lambda^*$ of $\mathcal D_\lambda$ is given by $\mathcal D_\lambda^*T = - (\Div_\lambda T)^\sharp$, and can be used to construct the self-adjoint elliptic operator $L_\lambda := D_\lambda^*D_\lambda$, which is called the \Defn{vector Laplace operator}.
If
\begin{equation*}
\sigma = \rho^{-1}\B_{\rho^2\lambda}(\rho) + \nu + \mathcal D_\lambda W,
\end{equation*}
where
\begin{equation}
\label{FirstVL}
L_{\lambda} W 
=  \Div_\lambda(\rho^{-1}\B_{\rho^2\lambda}(\rho) + \nu)^\sharp
-j - \mathcal E \times_\lambda \mathcal B,
\end{equation}
then \eqref{NewMomentumForSigma} is satisfied.
The proof of Theorem \ref{ExistenceProjection} shows that $\rho \mathcal D_\lambda W$ vanishes at $\partial M$.
It then follows that $\sigma$ satisfies \eqref{sigmaBC}, and thus that the corresponding seed data is shear-free.

\section{PDE results for implementing the conformal method}
\label{PDEResults}

We now gather several results concerning partial differential equations needed in order to carry out the procedure described above.  
These results, in turn, rely on the following Fredholm and regularity result of \cite{WAH-Preliminary}.

\begin{theorem}[Theorems 1.6 and A.14 of \cite{WAH-Preliminary}]
\label{Fred}
Suppose $g\in \WAH^{k,\alpha;1}$ for $k\geq 2$ and $\alpha \in (0,1)$, and that $\mathcal P$ is a second-order, linear, elliptic, formally self-adjoint, geometric operator obtained from $g$, and that there exist a compact set $K\subset M$ and a constant $C>0$ such that
\begin{equation}
\label{BasicL2Estimate}
\|u\|_{L^2(M)}\leq C\|\mathcal P u\|_{L^2(M)} 
\quad\text{ for all }\quad
 u\in C^\infty_c(M\setminus K).
\end{equation} 
Then the indicial radius $R$ of $\mathcal P$ is positive and 
$$\mathcal P\colon C^{k,\alpha}_\delta(M)\to C^{k-2,\alpha}_\delta(M)$$ is Fredholm of index zero if $|\delta-\tfrac{n}{2}|<R$.
Furthermore, the kernel is equal to the $L^2$ kernel of $\mathcal P$.

If, in addition, the metric $g$ is polyhomogeneous, then any solution $u$ to
$
\mathcal P u = f
$
is polyhomogeneous provided $f$ is polyhomogeneous.
\end{theorem}

\subsection{The vector Laplacian}
We first study the vector Laplacian, the invertibility of which fails if there exist non-trivial conformal Killing vector fields.
The following shows that such vector fields cannot vanish to second order at $\partial M$.

\begin{lemma}
\label{TrivialCKF}
Suppose $\lambda \in \WAH^{3,\alpha;1}$, and $X$ is a conformal Killing vector field on $(M,\lambda)$ such that $X\in C^{3,\alpha}_\delta(M)$ for some  $\alpha >0$ and $\delta>1$.
Then $X$ vanishes identically on $M$.
\end{lemma}

\begin{proof}
Let $\bar \lambda = \rho^2 \lambda$; for the purposes of this proof we denote by $\bar\nabla$ the Levi-Civita connection associated to $\bar \lambda$.

We adapt to the present setting the argument of \cite{Christodoulou-BoostProblem}, where it is observed that a conformal Killing vector field $X$ satisfies
\begin{equation}
\label{Christodoulou-ConformalKillingIdentity}
\bar\nabla^3X = R_0 \cdot \bar\nabla X + R_1\cdot X,
\end{equation}
for certain tensors $R_0$ linear in $\Riem^{(1,3)}[\bar \lambda]$, and $R_1$ linear in $\bar\nabla\Riem^{(1,3)}[\bar \lambda]$.
The strategy is to examine the set $Z\subset \bar M$ where $X$ vanishes together with $\bar{\nabla} X$ and $\bar{\nabla}^2X$. We start by showing that $Z$ is non-empty. 

Let $\mathcal C_a = \{\rho\leq a\}$ be a collar neighborhood of $\partial M$ and note that for a tensor $T$ of weight $r$ we have
\begin{equation*}
|T|_{\bar h} = \rho^{-r}|T|_h \leq \|T\|_{C^{0,0}_r(\mathcal C_a)}.
\end{equation*}
Thus if $T\in C^{0,0}_{r+\epsilon}(M)$ 
then $T\in C^0(\bar M)$ and is $\mathcal O(\rho^\epsilon)$; furthermore, the functions $T_{i_1\dots 1_p}^{j_1\dots j_q}$ which describe $T$ in coordinates satisfy
\begin{equation}
\label{GenericTensorComponentEstimate}
|T_{i_1\dots 1_p}^{j_1\dots j_q}|\leq C \rho^{\epsilon} \| T\|_{C^{0,0}_{r+\epsilon}(\mathcal C_a)}.
\end{equation}
Since the difference tensor $\nabla - \bar\nabla$ maps $C^{k,\alpha}_\delta(M)$ to itself, we have $\bar\nabla X \in C^{2,\alpha}_\delta(M)$ and $\bar\nabla{}^2X\in C^{1,\alpha}_\delta(M)$.
Consequently, because $X$ is a tensor of weight $r=-1$, we have the following pointwise estimates in background coordinates:
\begin{equation}\label{X-coordinate-ests}
\begin{aligned}
\big{|}X^k\big{|}&\leq C\rho^{\delta+1}\|X\|_{C^0_\delta(\mathcal C_a)}\\
\big{|}\left(\bar{\nabla}X\right)_{i}^k\big{|}&\leq C\rho^{\delta}\|\bar{\nabla}X\|_{C^{0}_\delta(\mathcal C_a)}\leq C\rho^{\delta}\|X\|_{C^1_\delta(\mathcal C_a)}\\
\big{|}\big(\bar{\nabla}^2X\big)_{ij}^k\big{|}&\leq C\rho^{\delta-1}\|\bar{\nabla}^2X\|_{C^{0}_\delta(\mathcal C_a)}\leq C\rho^{\delta-1}\|X\|_{C^2_\delta(\mathcal C_a)}.
\end{aligned}
\end{equation}
Because $\delta>1$, it follows that $X$, $\bar{\nabla} X$ and $\bar{\nabla}^2 X$ all vanish along $\partial M$. In other words, $\partial M\subset Z$.

Since $Z$ is closed and non-empty, it remains to show that $Z$ is open as well.
We proceed by first showing that $X$ vanishes on $\mathcal C_a$ for sufficiently small $a>0$.
 Let $\bar\Gamma$ represent the Christoffel symbols of $\bar\nabla$ in background coordinates near $\partial M$. As $\lambda \in \WAH^{3,\alpha;1}$ we know that $\bar\Gamma$ is bounded. 
Note that 
\begin{equation}
\partial_\rho X^k = \left(\bar\nabla_{\partial_\rho} X\right)^k - \left(\bar\Gamma\cdot X\right)^k.
\end{equation}
Integrating from $\rho=0$, where $X$ vanishes, and using \eqref{X-coordinate-ests} we obtain
\begin{equation}
\begin{aligned}
|X^k|&\le \int_0^{\rho}
\left(\,
\left|\left(\bar\nabla_{\partial_\rho} X\right)^k\right|  
+ \left|\left(\bar\Gamma\cdot X\right)^k\right|
\,\right)\,d\rho' \\
&\le C\int_0^{\rho} \left( (\rho')^\delta \|\bar\nabla X\|_{C^0_\delta(\mathcal C_a)}
+(\rho')^{\delta+1} \| X\|_{C^0_\delta(\mathcal C_a)}\right)\,d\rho'\\
&\le C\rho^{\delta+1}\left(\|\bar\nabla X\|_{C^0_\delta(\mathcal C_a)}
+\rho \| X\|_{C^0_\delta(\mathcal C_a)}\right),
\end{aligned}
\end{equation}
and thus 
\begin{equation}
\begin{aligned}
\|X\|_{C^0_\delta (\mathcal C_a)}
&\le C \sup_{k,\mathcal C_a} \big( \rho^{-\delta-1}|X^k|\big)
\le C\left(\|\bar\nabla X\|_{C^0_\delta(\mathcal C_a)}
+\rho \| X\|_{C^0_\delta(\mathcal C_a)}\right).
\end{aligned}
\end{equation}
Absorbing the final term, we obtain 
\begin{equation}\label{Xestimate}
\|X\|_{C^0_\delta (\mathcal C_a)}\leq C\|\bar\nabla X\|_{C^0_\delta (\mathcal C_a)}
\end{equation}
provided $a$ is small.
Using analogous integration of
\begin{equation}
\partial_\rho (\bar\nabla X)_i^k =\left(\bar\nabla_{\partial_\rho}(\bar\nabla X)\right)_i^k + \left(\bar\Gamma\cdot \bar\nabla X\right)_i^k
\end{equation}
we find that for sufficiently small $a$ we have 
\begin{equation}\label{barnablaX}
\|\bar{\nabla}X\|_{C^0_\delta (\mathcal C_a)}\leq C\|(\bar\nabla^2 X)\|_{C^0_\delta (\mathcal C_a)}.
\end{equation}
Combining 
\eqref{Xestimate} and \eqref{barnablaX} we obtain 
\begin{equation}
\|X\|_{C^2_\delta (\mathcal C_a)}\leq C \|\bar\nabla^2X\|_{C^0_\delta(\mathcal C_a)}.
\end{equation}

Next, note that \eqref{Christodoulou-ConformalKillingIdentity} implies 
\begin{equation}
\partial_\rho (\bar\nabla^2 X)_{ij}^k = (\bar\Gamma \cdot \bar\nabla^2 X)_{ij}^k + (R_0 \cdot \bar\nabla X)_{ij}^k + (R_1\cdot X)_{ij}^k.
\end{equation}
Using \eqref{GenericTensorComponentEstimate} we see that  the coefficients of $R_0$ and $R_1$ in background coordinates are $\mathcal O(\rho^{-1})$ and $\mathcal O(\rho^{-2})$, respectively. 
Thus using \eqref{X-coordinate-ests} again and integrating from $\rho =0$, we obtain
\begin{equation}
|(\bar\nabla^2X)_{ij}^k| 
\leq   C\rho^\delta \|X\|_{C^2_\delta(\mathcal C_a)}.
\end{equation}
Thus for sufficiently small $a>0$ we have $\|\bar\nabla^2X\|_{C^0_\delta(\mathcal C_a)} \leq C \rho \|X\|_{C^2_\rho(\mathcal C_a)}$.
When combined with $\|X\|_{C^2_\delta (\mathcal C_a)}\leq C \|\bar\nabla^2X\|_{C^0_\delta(\mathcal C_a)}$, this implies that $X$ vanishes identically on some collar neighborhood $\{\rho\le a\}$ of $\partial M$.

Next, suppose that $p\in Z$  is a point where  $\rho>a/2$, and consider the restriction of $X$ to the open geodesic ball $B(p;\epsilon)$ with radius $\epsilon$ and center $p$. The fundamental theorem of calculus implies that there exists some constant $C$, independent of $\epsilon$, such that
\begin{equation}
\label{Christodoulou-TaylorBound}
\|X\|_{C^2(B(p;\epsilon))} \leq \epsilon C\|\bar\nabla^3X\|_{C^0(B(p;\epsilon))}.
\end{equation}
Since $\rho(p)>a/2$, the restriction to $B(p;\epsilon)$ of $\Riem^{(1,3)}[\bar \lambda]$ and $\bar\nabla\Riem^{(1,3)}[\bar \lambda]$, and hence of the coefficients $R_0$ and $R_1$ in \eqref{Christodoulou-ConformalKillingIdentity}, are bounded by some constant $C$.
For such $\epsilon$, the identity \eqref{Christodoulou-ConformalKillingIdentity} implies \begin{equation}
\|\bar\nabla^3X\|_{C^0(B(p;\epsilon))}
\leq C \|X\|_{C^1(B(p;\epsilon))}.
\end{equation}
Combining this estimate with \eqref{Christodoulou-TaylorBound}, we conclude that for sufficiently small $\epsilon$, the vector field $X$ must vanish on $B(p;\epsilon)$, so $B(p;\epsilon)\subset Z$. This completes our proof. 
\end{proof}

We now address the invertibility of the vector Laplacian $L_\lambda$.

\begin{proposition}
\label{SolveVectorLaplacian}
Suppose that $\lambda \in \WAH^{k,\alpha;1}$ for  $k\geq 3$ and $\alpha\in (0,1)$.
Let $\delta\in (-1,3)$.
Then for each vector field $Y\in C^{k-2,\alpha}_\delta(M)$ there exists a unique vector field $W\in C^{k,\alpha}_\delta(M)$ such that
\begin{equation}
\label{VectorLaplacianEquation}
L_{\lambda} W =Y.
\end{equation}
Furthermore, there exists a constant $C>0$ such that for all $W\in C^{k,\alpha}_\delta(M)$ we have
\begin{equation}
\label{VL-InvertibilityEstimate}
\|W\|_{C^{k,\alpha}_\delta(M)}\leq C \|L_\lambda W\|_{C^{k-2,\alpha}_\delta(M)}.
\end{equation}
Finally, if $\rho^2\lambda\in C^2_\phg(\bar M)$ and $\rho^{-3}Y\in C^0_\phg(\bar M)$, then $\rho^{-3}W\in C^0_\phg(\bar M)$ and $\mathcal D_\lambda W\in C^0_\phg(\bar M)$.
\end{proposition}

\begin{proof} 
It is straightforward to see that the vector Laplacian $L_\lambda$ is a formally self-adjoint, elliptic, geometric operator.
That the basic estimate \eqref{BasicL2Estimate} holds is the content of Lemma 3.15 of \cite{Andersson-EllipticSystems}. 
Furthermore, in \cite{Lee-FredholmOperators} the indicial radius of $L_\lambda$ is shown to be $2$; thus we may invoke Theorem \ref{Fred} to conclude that
\begin{equation}
\label{WhereDoesVectorLaplacianGo?}
L_\lambda\colon C^{k,\alpha}_\delta(M;TM)\to C^{k-2,\alpha}_\delta(M;TM)
\end{equation}
is Fredholm of index zero.

To show that \eqref{VectorLaplacianEquation} admits a unique solution, we argue that the kernel of $L_\lambda$ is trivial.
Supposing that $X\in C^{k,\alpha}_\delta(M)$ and that $L_\lambda X=0$,  elliptic regularity and Lemma 4.6 of \cite{WAH-Preliminary} imply that $X\in C^{k,\alpha}_2(M)$.
By Lemma 3.6(b) of \cite{Lee-FredholmOperators}, $X$ is in the Sobolev space $H^{2,2}(M)$.
For all $V\in C^\infty_c(M)$ we have $\|\mathcal D_\lambda V\|_{L^2}^2 = (V, L_\lambda V)_{L^2}$; due to the density of compactly supported fields, this identity extends by continuity to all vector fields in  $H^{2,2}(M)$.
Thus  $\|\mathcal D_\lambda X\|_{L^2(M)}=0$ and $X$ is a conformal Killing vector field on $(M,g)$.
Since $X\in C^{3,\alpha}_2(M)$,
Lemma \ref{TrivialCKF} tells us that $X$ must be identically zero.
Therefore the $C^{k,\alpha}_\delta$ kernel of $L_\lambda$ is trivial and the mapping \eqref{WhereDoesVectorLaplacianGo?} is a bijection.
The continuity of $L_\lambda$, together with the closed graph theorem, yields \eqref{VL-InvertibilityEstimate}.

When  $\rho^2g \in C^2_\phg(\bar M)$, Theorem \ref{Fred} implies that the solution is polyhomogeneous when $Y$ is.
If $\rho^{-3}Y\in C^0_\phg(\bar M)$, then Lemma A.5 of \cite{WAH-Preliminary} implies that we may furthermore choose $\gamma$ such that  $\rho^{-3}Y\in C^{0,\gamma}_\phg(\bar M)$, and hence
$Y\in C^{k,\alpha}_2(M)$ for all $k\in\mathbb N_0$.
Subsequently, we find that $\rho^{-3}W\in C^{0}_\phg(\bar M)$.
Finally, it follows from a direct computation that $\mathcal D_\lambda W\in C^0_\phg(\bar M)$.
\end{proof}

\begin{remark}
\label{WhyK3}
The maps \eqref{WhereDoesVectorLaplacianGo?} are Fredholm of index zero even if $\lambda\in \WAH^{2,\alpha;1}$.
The hypothesis $k\geq 3$ is used only to show that the kernel is trivial.
\end{remark}

\subsection{Scalar equations}
\label{S-ScalarEquations}
We record here several results for scalar equations that follow directly from the results of \cite{WAH-Preliminary}.

\begin{proposition}[Proposition 6.1 of \cite{WAH-Preliminary}]
\label{PerturbedPoisson}
Suppose that $g \in \WAH^{k,\alpha;1}$ for  $k\geq 2$ and $\alpha\in (0,1)$.
Suppose also that $\kappa\in C^{k-2,\alpha}_\sigma(M)$ for some $\sigma>0$, and that  $c$ is a constant satisfying $c>-1$ and $c-\kappa\geq 0$.
Then so long as 
\begin{equation}
\label{PP-Radius}
\left|\delta-1\right|<\sqrt{1+c},
\end{equation}
the map
\begin{equation*}
\Delta_g - (c-\kappa)\colon C^{k,\alpha}_\delta(M)\to C^{k-2,\alpha}_\delta(M)
\end{equation*}
is invertible.

Furthermore, if $\rho^2g\in C^2_\phg(\bar M)$, $\rho^{-\nu}f\in C^0_\phg(\bar M)$ for some $\nu >1-\sqrt{1+c}$, and $\kappa$ is a polyhomogeneous function \(which necessarily vanishes on $\partial M$\), then the unique function $u\in C^{2,\alpha}_\delta(M)$ such that
\begin{equation}
\label{PP}
\Delta_\lambda u + (\kappa-c) u = f
\end{equation} 
is polyhomogeneous and satisfies the following boundary regularity conditions:
\begin{itemize}
\item If $\nu > 1+\sqrt{1+c}$, then $\rho^{-1-\sqrt{1+c}}\, u \in C^0_\phg(\bar M)$.
\item If $|\nu - 1|<\sqrt{1+c}$, then $\rho^{-\nu} u\in C^0_\phg(\bar M)$.
\item If $\nu = 1+\sqrt{1+c}$, then $\rho^{-\mu} u \in C^0_\phg(\bar M)$ for all $\mu< 1+\sqrt{1+c}$.
\end{itemize}
\end{proposition}

The following consequence of Proposition \ref{PerturbedPoisson} is related to the Helmholtz decomposition for vector fields.
\begin{corollary}[Helmholtz splitting]
\label{Helmholtz}
Suppose that $\lambda \in \WAH^{k,\alpha;1}$ for  $k\geq 2$ and $\alpha\in (0,1)$, and that $\delta\in (0,2)$.
Then any vector field $X\in C^{k-1,\alpha}_\delta(M)$ can be uniquely written as
\begin{equation}
X = \grad_\lambda u + Y,
\end{equation}
where the vector field $Y$ satisfies $\Div_\lambda Y=0$ and the function $u$ satisfies $u\in C^{k,\alpha}_\delta(M)$.

Furthermore, if $\rho^2\lambda\in C^2_\phg(\bar M)$ and $\rho^{-3}X\in C^0_\phg(\bar M)$, then $\rho^{-3} Y \in C^0_\phg(\bar M)$, $\rho^{-2}u \in C^0_\phg(\bar M)$, and $\rho^{-3}\grad_\lambda u\in C^0_\phg(\bar M)$.
\end{corollary}

\begin{proof}
Since $Y$ is divergence-free precisely if $\Delta_\lambda u = \Div_\lambda X$, the existence of $u$ and $Y$ is immediate.

Suppose now that $\bar X:=\rho^{-3}X\in C^0_\phg(\bar M)$ and that $\bar\lambda:= \rho^{2}\lambda\in C^2_\phg(\bar M)$.
By Lemma A.5 of \cite{WAH-Preliminary} we have $\bar X\in C^{0,\gamma}_\phg(\bar M)$ and $\bar \lambda \in C^{2,\gamma}_\phg(\bar M)$ for some $\gamma>0$.
Since $\Nabla{\bar\lambda}\colon C^{0,\gamma}_\phg(\bar M)\to \rho^{\gamma -1} C^0_\phg(\bar M)$ we have $f = \Div_\lambda X = \rho^3 \Div_{\bar\lambda}\bar X\in \rho^{2+\gamma}C^0_\phg(\bar M)$.
This implies the claimed regularity of $Y$ and of $\grad_\lambda u$.
\end{proof}

We now address the solvability of the Lichnerowicz equation \eqref{LichPhi}, recalling the following from \cite{WAH-Preliminary}.

\begin{proposition}[Proposition 6.4 of \cite{WAH-Preliminary}]
\label{SolveLichPhi}
Suppose that $\lambda \in \WAH^{k,\alpha;1}$ for  $k\geq 2$ and $\alpha\in (0,1)$.
Suppose furthermore that $A$ and $B$ are non-negative functions with $A,B\in C^{k-2,\alpha}_1( M)$.
\begin{enumerate}
\item 
\label{part:FirstSolveLichPhi}
Then there exists a unique function $\phi$ with $\phi-1\in C^{k,\alpha}_1(M)$ that satisfies
\begin{equation}
\label{SimpleLichPhi}
\begin{gathered}
\Delta_\lambda\phi = \frac18R[\lambda]\phi - A\phi^{-7} - B\phi^{-3} + \frac34\phi^5,
\\
\left.\phi\right|_{\partial M} =1,\qquad \phi>0.
\end{gathered}
\end{equation}
The regularity on $\phi$ implies $\phi^4\lambda \in \WAH^{k,\alpha;1}$.
\end{enumerate}
Furthermore:
\begin{enumerate}[resume]
\item
\label{part:simple-lich-better-r}
 If $\lambda \in \WAH^{k,\alpha;2}$, $\R[\lambda]+6\in C^{k-2,\alpha}_2(M)$, and $A,B\in C^{k-2,\alpha}_2(M)$, then $\phi-1\in C^{k,\alpha}_2(M)$ and hence $\phi^4 \lambda \in \WAH^{k,\alpha;2}$.
\item If $\rho^2\lambda \in C^2_\phg(\bar M)$ and $\rho^{-2}A, \rho^{-2}B\in C^0_\phg(\bar M)$,  then $\rho^2 \phi^4\lambda \in C^2_\phg(\bar M)$.
\end{enumerate}

\end{proposition}

\section{Continuity results for the conformal method}
\label{PDE-Continuity}

In Section \ref{Section-PandE} below, we frame the conformal method as maps taking free data sets to seed data sets, and taking seed data sets to initial data sets.
Here we establish a collection of results that we use to show the continuity of these maps.
We first recall the following definitions.
Suppose $X$ and $Y$ are normed spaces, and $F\colon X\to Y$.
We say that $F$ is \Defn{locally bounded} if for each $x\in X$ there exist constants $\epsilon_x>0$ and $C_x>0$ such that
\begin{equation*}
\| F(x^\prime) \|_Y \leq C_x
\quad\text{ if }\quad
\|x-x^\prime \|_X \leq \epsilon_x.
\end{equation*}
We furthermore say that $F$ is \Defn{locally Lipschitz} if
for each $x\in X$ there exist constants $\epsilon_x>0$ and $C_x>0$ such that
\begin{equation*}
\| F(x_1)- F(x_2)\|_Y \leq C_x \|x_1 - x_2\|_X
\quad\text{ if }\quad
\|x-x_i \|_X \leq \epsilon_x, \quad i=1,2.
\end{equation*}
Note that all locally Lipschitz mappings are locally bounded, and that all continuous linear  maps are locally (in fact, globally) Lipschitz.

The facts that $\mathscr C^{k,\alpha;m}(M)$ is an algebra and $\rho\in C^\infty(\bar M)$ imply the following.
\begin{lemma}
\label{lemma:BasicLipschitzFacts}

Suppose $0\leq m \leq k$, $0\leq m^\prime \leq k^\prime$, and $\alpha, \alpha^\prime \in [0,1)$.
\begin{enumerate}
\item If $F_1$ and $F_2$ are locally Lipschitz maps $\colon \mathscr C^{k,\alpha;m}(M) \to \mathscr C^{k^\prime, \alpha^\prime, m^\prime}(M)$, then so is $F_1\otimes F_2$.

\item Contraction is a continuous linear map $\mathscr C^{k,\alpha;m}(M) \to \mathscr C^{k,\alpha;m}(M)$.

\item Multiplication by $\rho$ is a continuous linear map $\mathscr C^{k,\alpha;m}(M) \to \mathscr C^{k,\alpha;m+1}(M)$.

\item For any $l\geq 1$, $u\mapsto u\otimes \bar\nabla{}^l\rho$ is a continuous linear map $\mathscr C^{k,\alpha;m}(M)\to \mathscr C^{k,\alpha;m}(M)$.

\item The map  $\bar\nabla \colon \mathscr C^{k+1,\alpha;m+1}(M) \to \mathscr C^{k,\alpha;m}(M)$ is a continuous linear map.

\end{enumerate}
\end{lemma}

We now show that the construction of various geometric tensor fields from weakly asymptotically hyperbolic metrics is locally Lipschitz.

\begin{lemma}
\label{lemma:cont-inverse}
Suppose $0\leq m \leq k$ and $\alpha \in[0,1)$.
The map $\lambda \mapsto (\bar \lambda)^{-1}$ is a locally Lipschitz continuous map from $\WAH^{k,\alpha;m}$ to $\mathscr C^{k,\alpha;m}(M)$.
\end{lemma}
\begin{proof}
First note that if $\lambda\in \WAH^{k,\alpha;m}$ then $\|(\bar \lambda_1)^{-1}\|_{C^{k,\alpha}_{-2}(M)}$ is bounded.
Thus the lemma follows from estimating in $C^{k,\alpha}_{-2}(M)$ the series expansion of
\begin{equation}
(\bar \lambda_1)^{-1} - (\bar \lambda_2)^{-1} 
= (\bar \lambda_1)^{-1}\left(1-\left[1-(\bar \lambda_1-\bar \lambda_2)(\bar \lambda_1)^{-1}\right]^{-1}\right),
\end{equation}
which converges uniformly and absolutely when $\|\bar\lambda_1 - \bar\lambda_2 \|_{C^0(M)}$ is small.
\end{proof}

We estimate the connections of weakly asymptotically hyperbolic metrics by comparing them to the connection $\bar\nabla$ associated to the background metric $\bar h$.
In doing this, we  use the notation $D[\bar\lambda] = \Nabla{\bar\lambda} - \bar\nabla$ for the difference tensor.

\begin{lemma}
\label{lemma:BarConnectionLip}
Suppose $1\leq m \leq k$ and $\alpha \in[0,1)$.
The map $\lambda \mapsto D[\bar\lambda]$ is locally Lipschitz continuous from $\WAH^{k,\alpha;m}$ to $\mathscr C^{k-1,\alpha;m-1}(M)$.
\end{lemma}

\begin{proof}
As the tensor $D[\bar\lambda]$ is a sum of contractions of terms of the form $(\bar\lambda)^{-1}\otimes \bar\nabla\, \bar\lambda$, the proof follows directly from Lemma \ref{lemma:BasicLipschitzFacts}.
\end{proof}

\begin{remark}
\label{rmk:BasicNoBar}
Suppose $0\leq m \leq k$ and $\alpha \in [0,1)$.
Arguments analogous to those above show that $\lambda \mapsto \lambda^{-1}$ is a locally Lipschitz map from $\WAH^{k,\alpha;m}$ to $C^{k,\alpha}(M)$ and that $\lambda \mapsto \Nabla{\lambda} - \nabla$ is a locally Lipschitz map from $\WAH^{k+1,\alpha;m}$ to $C^{k,\alpha}(M)$, using the fact that $\nabla$ is a continuous map from $C^{k+1,\alpha}(M)$ to $C^{k,\alpha}(M)$.
\end{remark}

We also require the following.

\begin{lemma}
\label{lemma:cont-hess-etc}
Suppose $1\leq m \leq k$ and $\alpha \in [0,1)$.
The following maps are locally Lipschitz:
\begin{enumerate}
\item 
\label{part:hessian}
$\lambda \mapsto \Hess_{\bar\lambda}(\rho)$, as a map from $\WAH^{k,\alpha;m}$ to $ \mathscr C^{k-1,\alpha;m-1}(M)$,

\item $\lambda \mapsto \B_{\bar\lambda}(\rho)$, as a map from $\WAH^{k,\alpha;m}$ to $ \mathscr C^{k-1,\alpha;m-1}(M)$,
and 

\item
$\lambda \mapsto \Div_\lambda(\rho^{-1}\B_{\bar\lambda}(\rho))$, as a map from $\WAH^{k,\alpha;m}$ to $  C^{k-2,\alpha}_{2}(M)$ provided $2\leq m \leq k$.

\end{enumerate}
\end{lemma}

\begin{proof}
The first claim follows from writing $\Hess_{\bar\lambda}(\rho) = D[\bar\lambda]\D\rho + \Hess_{\bar h}(\rho)$ and Lemma \ref{lemma:BarConnectionLip}, while the second claim follows from the first and from the fact that each term in $\B_{\bar\lambda}(\rho)$ is a contraction of
\begin{equation*}
(\bar\lambda)^{-1}\otimes
(\bar\lambda)^{-1}\otimes
(\bar\lambda)^{-1}\otimes
(\bar\lambda)^{-1}\otimes
\bar\lambda\otimes
\D\rho\otimes
\D\rho\otimes
\D\rho\otimes
\D\rho\otimes
\Hess_{\bar\lambda}(\rho).
\end{equation*}
For the third claim, we note that Proposition \ref{B-BasicProperties}\eqref{B-TransverseProperty} and the formula for divergence under conformal change imply
\begin{equation}
\Div_\lambda (\rho^{-1}\B_{\bar\lambda}(\rho)) 
= \rho^{-1}\Div_\lambda \B_{\bar\lambda}(\rho) 
= \rho
\Div_{\bar\lambda}\B_{\bar\lambda}(\rho),
\end{equation}
which is a contraction of $\rho(\bar\lambda)^{-1}\otimes \Nabla{\bar\lambda} \B_{\bar \lambda}(\rho)$.
But $\lambda \mapsto (\bar\lambda)^{-1}\otimes \Nabla{\bar\lambda} \B_{\bar \lambda}(\rho)$ is a locally Lipschitz map from $\WAH^{k,\alpha;m} $ to $ \mathscr C^{k-2,\alpha; m-2}(M)$.
As tensor fields of weight $1$ in $ \mathscr C^{k-2,\alpha; m-2}(M)$ are in $C^{k,\alpha}_1(M)$, the proof is complete.
\end{proof}

We now turn to the analysis of scalar curvature, viewed as a function of the metric.
For  $k\geq 2$,  $\alpha\in[0,1)$, and $1\leq m \leq k$, we denote by $\HAR^{k,\alpha;m}$ the collection of those metrics $\lambda\in\WAH^{k,\alpha;m}$ such that $\R[\lambda]+6\in C^{k-2,\alpha}_m(M)$.
 (In particular, if $\lambda\in \HAR^{k,\alpha;m}$ then $\R[\lambda]+6 = \mathcal O(\rho^m)$ as $\rho\to 0$.)
 It is a direct consequence of Theorem 1.2(c) in \cite{WAH-Preliminary} that $\HAR^{k,\alpha;1} = \WAH^{k,\alpha;1}$.
However, $\HAR^{k,\alpha;2}$ is a proper subset of $\WAH^{k,\alpha;2}$.
In \cite{WAH-Preliminary}, we establish the following.
\begin{proposition}[Lemma 6.5 of \cite{WAH-Preliminary}]
\label{prop:DeformToHAR}
Suppose $\lambda\in \WAH^{k,\alpha;2}$ with $k\ge 2$ and $\alpha\in [0,1)$.
Then the conformal class of $\lambda$ contains a representative $\tilde \lambda\in \HAR^{k,\alpha;2}$.
Furthermore, $\lambda\mapsto \tilde \lambda$ is a locally Lipschitz map from $\WAH^{k,\alpha;2}$ to $\HAR^{k,\alpha;2}$.
\end{proposition}

\begin{remark}
In fact, the proof of Proposition \ref{prop:DeformToHAR} in \cite{WAH-Preliminary} shows that for each metric $\lambda\in \WAH^{k,\alpha;2}$ there exists a positive function $\theta \in \mathscr C^{k,\alpha;2}(M)$ satisfying $\theta-1\in C^{k,\alpha}_1(M)$ and such that  $\tilde\lambda = \theta^{4}\lambda \in \HAR^{k,\alpha;2}$.
\end{remark}

Estimates on scalar curvature are used in the analysis of the Lichnerowicz equation.
Due to the conformal covariance of the Lichnerowicz equation (see Lemma \ref{lemma:LichInvariance}), we only need to establish continuity of $\lambda\mapsto \R[\lambda]+6$ for metrics in $\HAR^{k,\alpha;m}$. 

\begin{lemma}
\label{lemma:cont-R}
Let $k\geq 2$ and $\alpha\in [0,1)$.
The map $\lambda \mapsto \R[\lambda]+6$ is locally Lipschitz as a map from $\HAR^{k,\alpha;2}$ to $ C^{k-2,\alpha}_{2}(M)$.
\end{lemma}

\begin{proof}
See Remark 3.2 of \cite{WAH-Preliminary}.
\end{proof}

We now discuss continuity for geometric differential operators.
Recall from \cite{Lee-FredholmOperators} that a linear differential operator $\mathcal P= \mathcal P[\lambda]$ is \Defn{geometric} of order $l$ if $\mathcal P[\lambda]u$ can be expressed as sums of contractions of terms of the form 
\begin{equation*}
\Nabla{\lambda}{}^iu 
\otimes \Nabla{\lambda}{}^{k_1}\Riem[\lambda] \otimes\dots \otimes \Nabla{\lambda}{}^{k_j}\Riem[\lambda]
\otimes(\otimes ^p \lambda^{-1}) \otimes (\otimes^q \lambda) \otimes (\otimes^s dV_{\lambda})
\end{equation*}
with $0\leq i\leq l$ and $0\leq k_t \leq l-i -2$.
(The volume form is permitted only if $M$ is oriented.)

Geometric operators are locally Lipschitz in the following sense.
\begin{proposition}
\label{prop:GeoOpCont}
Let $0\leq m \leq k$, $\alpha \in [0,1)$, and $\delta\in \mathbb R$.
Suppose $\mathcal P$ is a geometric operator of order $l$ with $l\leq k$.
Then for each $\lambda \in \WAH^{k,\alpha;m}$ there exist positive constants $C$ and $\epsilon$, depending on $\lambda$, such that for all $u\in C^{k,\alpha}_\delta(M)$ we have
\begin{equation*}
\| \mathcal P[\lambda_1] u - \mathcal P[\lambda_2]u\|_{C^{k-l,\alpha}_\delta(M)}
\leq C 
\gNorm{\bar \lambda_1 - \bar\lambda_2}_{k,\alpha;m} 
\|u\|_{C^{k,\alpha}_\delta(M)}
\end{equation*}
so long as $\gNorm{\bar \lambda_i - \bar\lambda}_{k,\alpha;m}\leq \epsilon$ for $i=1,2$.
\end{proposition}

\begin{proof}
It suffices to take $m=0$ and to show that the following are locally Lipschitz:
\begin{enumerate}
\item the map $\WAH^{k,\alpha;0}\to C^{k,\alpha}(M)$ given by $\lambda\mapsto \lambda^{-1}$,
\item the map $\WAH^{k,\alpha;0}\to C^{k,\alpha}(M)$ given by $\lambda\mapsto dV_{\lambda}$,
\item the map $\WAH^{k,\alpha;0}\to C^{k-1,\alpha}(M)$ given by $\lambda\mapsto \Nabla{\lambda} - \nabla$,
and
\item the map $\WAH^{k,\alpha;0}\to C^{k-2,\alpha}(M)$ given by $\lambda\mapsto \Riem[\lambda]$.
\end{enumerate}
The continuity of the first three maps follows from  direct computation and from Remark \ref{rmk:BasicNoBar}.
The continuity of $\Riem[\lambda]$ follows from the analysis of terms of the form
\begin{equation*}
\lambda^{-1}\otimes \lambda^{-1} \otimes \nabla^2\lambda
\quad\text{ and }\quad
\lambda^{-1} \otimes \lambda^{-1} \otimes \nabla \lambda \otimes \nabla\lambda.
\end{equation*}
The result now follows from Lemma \ref{lemma:BasicLipschitzFacts}.
\end{proof}

Proposition \ref{prop:GeoOpCont} allows us to establish local Lipschitz continuity for the various linear geometric operators $\mathcal P$ arising in the conformal method.
In the case that $\mathcal P$ is a second-order elliptic operator, such as the scalar Laplace operator $\Delta_\lambda$ or the vector Laplace operator $L_\lambda$, it is useful to obtain additional estimates in those spaces on which $\mathcal P$ is invertible.
An important observation is that in Proposition \ref{Fred}, the range of weights $\delta$ for which $\mathcal P$ is Fredholm of index zero is independent of the metric $\lambda$, but instead depends only on the algebraic structure of the operator; see \cite{Lee-FredholmOperators} (or \cite{WAH-Preliminary}) for additional details.
In fact, for both the scalar and vector Laplace operators, Propositions \ref{SolveVectorLaplacian} and \ref{PerturbedPoisson} imply that the operators are invertible for all metrics of sufficient regularity.
We now establish a certain type of  continuity for the inverses of these operators.  

\begin{proposition}
\label{prop:LinearContinuity}
Let $0\leq m \leq k$, $\alpha \in [0,1)$, and $\delta\in \mathbb R$. 
Suppose that $\mathcal P = \mathcal P[\lambda]$ is a second-order linear geometric elliptic operator such that for each $\lambda \in \WAH^{k,\alpha;m}$ and $f\in C^{k-2,\alpha}_\delta(M)$ there exists a unique solution $u\in C^{k,\alpha}_\delta(M)$ to 
$
\mathcal P[\lambda] u = f.
$
Then the map $(\lambda, f)\mapsto u$ is locally Lipschitz continuous as a map 
\begin{equation*}
\WAH^{k,\alpha;m} \times C^{k-2,\alpha}_\delta(M)
\to C^{k,\alpha}_\delta(M).
\end{equation*}
\end{proposition}

\begin{proof}
Fix $\lambda\in \WAH^{k,\alpha;m}$ and $f\in C^{k-2,\alpha}_\delta(M)$, and suppose for $i=1,2$ that $\lambda_i\in \WAH^{k,\alpha;m}$ and $f_i\in C^{k-2,\alpha}_\delta(M)$ 
satisfy
\begin{equation*}
\gNorm{\bar\lambda_i - \bar\lambda}_{k,\alpha;m}
+
\|f_i - f\|_{C^{k-2,\alpha}_\delta(M)}<\epsilon
\end{equation*}
for some $\epsilon>0$, chosen below depending on $\lambda$.
Suppose as well that $\mathcal P[\lambda_i]u_i = f_i$.

The invertibility of $\mathcal P[\lambda]$ implies that there exists $C>0$ such that for all $u\in C^{k,\alpha}_\delta(M)$ we have
\begin{equation*}
\| u \|_{C^{k,\alpha}_\delta(M)}
\leq C \|\mathcal P[\lambda]u \|_{C^{k-2,\alpha}_\delta(M)}.
\end{equation*}
Since
\begin{equation*}
\|\mathcal P[\lambda]u \|_{C^{k-2,\alpha}_\delta(M)} 
\leq
\|\mathcal P[\lambda_i]u \|_{C^{k-2,\alpha}_\delta(M)}
+
\|(\mathcal P[\lambda_i] -\mathcal P[\lambda])u \|_{C^{k-2,\alpha}_\delta(M)},
\end{equation*}
we may invoke Proposition \ref{prop:GeoOpCont} to see that there exists some $\epsilon_*>0$ such that for $\epsilon\in (0,\epsilon_*]$ we have
\begin{equation}
\label{UniformInvert}
\| u \|_{C^{k,\alpha}_\delta(M)}
\leq C_* \|\mathcal P[\lambda_i]u \|_{C^{k-2,\alpha}_\delta(M)}
\end{equation}
for all $u\in C^{k,\alpha}_\delta(M)$; here $C_*$ depends only on $\lambda$, $\epsilon_*$ and universal parameters.

We apply this latter estimate to the identity
\begin{equation*}
\mathcal P[\lambda_1](u_2-u_1)
= f_2-f_1
- (\mathcal P[\lambda_2]-\mathcal P[\lambda_1])u_2.
\end{equation*}
Proposition \ref{prop:GeoOpCont} implies that
\begin{equation*}
\| (\mathcal P[\lambda_2]-\mathcal P[\lambda_1])u_2\|_{C^{k-2,\alpha}_\delta(M)}
\leq C \gNorm{\lambda_2-\lambda_1}_{k,\alpha;m} \|u_2\|_{C^{k,\alpha}_\delta(M)},
\end{equation*}
while \eqref{UniformInvert} implies 
\begin{equation*}
\|u_2\|_{C^{k,\alpha}_\delta(M)}
\leq C_* \|f_2\|_{C^{k-2,\alpha}_\delta(M)}\leq C_* \left(\|f\|_{C^{k-2,\alpha}_\delta(M)} + \epsilon_*\right).
\end{equation*}
Thus 
\begin{equation*}
\| u_1 - u_2\|_{C^{k,\alpha}_\delta(M)}
\leq C \left(
\| f_1 - f_2\|_{C^{k-2,\alpha}_\delta(M)}
+ \gNorm{\lambda_1 - \lambda_2}_{k,\alpha;m}
\right),
\end{equation*}
with $C$ depending only on $\lambda$ and $f$.
\end{proof}

Proposition \ref{prop:LinearContinuity} provides continuity for the linear existence results in \S\ref{PDEResults}.
We now establish continuity of the solution map defined by Proposition \ref{SolveLichPhi}.
In view of Proposition \ref{prop:DeformToHAR}  and Lemma \ref{lemma:LichInvariance} it suffices to show the following.

\begin{proposition}
\label{prop:LichPhiCont}
Let $k\geq 2$, and $\alpha \in (0,1)$.
Then the map $(\lambda, A, B)\mapsto \phi-1$, where $\phi$ is the solution to \eqref{SimpleLichPhi} established by Proposition \ref{SolveLichPhi}\eqref{part:simple-lich-better-r}, is a continuous map from 
$
\HAR^{k,\alpha;2} \times 
C^{k-2,\alpha}_2(M) \times 
C^{k-2,\alpha}_2(M)
$ to $
C^{k,\alpha}_2(M),
$
where we restrict the domain to $A,B\geq 0$.
\end{proposition}

\begin{proof}
Let $\lambda_1 \in \HAR^{k,\alpha;2}$ and $A_1,B_1\in C^{k-2,\alpha}_2(M)$.
Suppose $\lambda_2\in \HAR^{k,\alpha;2}$ and $A_2, B_2\in C^{k-2,\alpha}_2(M)$ are such that
\begin{equation}
\label{LichCont-Apriori}
\gNorm{\bar\lambda_1 - \bar\lambda_2}_{k,\alpha;2}
+ \| A_1 - A_2\|_{C^{k-2,\alpha}_2(M)}
+ \| B_1 - B_2\|_{C^{k-2,\alpha}_2(M)}
\leq \delta
\end{equation}
for some $\delta>0$, chosen below with dependence on $\lambda_1$, $A_1$, and $B_1$.
For $i=1,2$, let $\phi_i$ be the corresponding solutions to 
\begin{equation}
\label{LichPhii}
\Delta_{\lambda_i}\phi_i = \frac18 R[\lambda_i]\phi_i - A_i \phi^{-7}_i - B_i \phi_i^{-3} + \frac34 \phi_i^5.
\end{equation}

In order to estimate $\phi_2-\phi_1$ we consider the function $u  := \phi_1^{-1}(\phi_2-\phi_1)$.
As $\phi_1\in C^{k,\alpha}_0(M)$ is positive and bounded below, and $\phi_2-\phi_1\in C^{k,\alpha}_2(M)$, 
we have $u \in C^{k,\alpha}_2(M)$.

We make use of the smooth remainder functions  $F_5,F_{-3},F_{-7}\colon (-1,\infty)\to \mathbb R$ that satisfy
\begin{equation*}
\begin{gathered}
t^{5}-t = 4(t-1) + (t-1)^2 F_5(t-1),
\\
t^{-3}-1 = -3(t-1) + (t-1)^2 F_{-3}(t-1),
\\
t^{-7}-1 = -7(t-1) + (t-1)^2 F_{-7}(t-1),
\end{gathered}
\end{equation*}

Let $\gamma = \phi_1^4\lambda_1$. We have
\begin{equation*}
\begin{aligned}
\Delta_\gamma u 
= \phi_1^{-5} \Big[  &-(\Delta_{\lambda_1}\phi_1)(u+1) 
+ (\Delta_{\lambda_2}\phi_2)
+\phi_1(\Delta_{\lambda_1} - \Delta_{\lambda_2})u
+((\Delta_{\lambda_1} - \Delta_{\lambda_2})\phi_1)u
\\
&
+2(\lambda_1^{-1} - \lambda_2^{-1})(\D\phi_1, \D u)
\Big].
\end{aligned}
\end{equation*}
Using \eqref{LichPhii} we find that $u$ satisfies
\begin{equation}
\label{LinearizedLichRatio}
\left(\mathcal P+\mathcal K\right)u
= f + Q(u);
\end{equation}
here the operators $\mathcal P$ and $\mathcal K$, and the functions $f$ and $Q$, are defined by
\begin{align*}
\mathcal Pu &= \Delta_\gamma u - (3+8\phi_1^{-12}A_1 + 4\phi_1^{-8} B_1)u,
\displaybreak[0]
\\
\mathcal Ku&= 
-\phi_1^{-4}(\Delta_{\lambda_1}- \Delta_{\lambda_2})u
-2\phi_1^{-5}[(\lambda_1)^{-1}-(\lambda_2)^{-1}](\D\phi_1,\D u)
\\&\qquad + \frac18\phi_1^{-4}\left(\R[\lambda_1]-\R[\lambda_2] \right)u
- \phi_1^{-5}(\Delta_{\lambda_1}\phi_1 - \Delta_{\lambda_2}\phi_1)u
\\&\qquad+ 7\phi_1^{-12}\left(A_1 - A_2 \right)u
+ 3 \phi_1^{-8} \left( B_1 - B_2\right)u,
\displaybreak[0]
\\
f &= \phi_1^{-5}(\Delta_{\lambda_1}\phi_1 - \Delta_{\lambda_2}\phi_1) 
-\frac18\phi_1^{-4}\left(R[\lambda_1]-R[\lambda_2] \right)
\\&\qquad
+\phi_1^{-12} (A_1 - A_2 ) + \phi_1^{-8}(B_1-B_2),
\displaybreak[0]
\\
Q(u) 
&= \frac34 u^2F_5(u) 
- \phi_1^{-12} A_2 u^2 F_{-7}(u)
-\phi_1^{-8} B_2 u^2 F_{-3}(u).
\end{align*}
Note that $\mathcal P$ satisfies the hypotheses of Proposition \ref{PerturbedPoisson}.

Lemma \ref{lemma:cont-R}, Proposition \ref{prop:GeoOpCont},
and \eqref{LichCont-Apriori} imply that 
for all $\delta$ sufficiently small we have
\begin{equation}
\label{f-Small}
\|f\|_{C^{k-2,\alpha}_2(M)}\leq C \delta
\end{equation}
and
\begin{equation}
\label{K-Small}
\|\mathcal K u \|_{C^{k-2,\alpha}_2(M)}
\leq C  \delta
\|u\|_{C^{k,\alpha}_2(M)}
\end{equation}
for all $u\in C^{k,\alpha}_2(M)$; here 
the constant $C$ depends on $\phi_1$, which is determined by $\lambda_1$, $A_1$, and $B_1$.

Since $8\phi_1^{-12}A_1+4\phi_1^{-8}B_1$ is in $ C^{k-1,\alpha}_2(M)$ and is nonnegative, Proposition \ref{PerturbedPoisson} implies that $\mathcal P$ is invertible as a map $C^{k,\alpha}_2(M)\to C^{k-2,\alpha}_2(M)$.
A Neumann-series argument 
using \eqref{K-Small}
shows that $\mathcal P + \mathcal K$ is invertible for sufficiently small $\delta$.
Thus  
\begin{equation}
\label{UniformInversion}
\|u\|_{C^{k,\alpha}_2(M)} \leq C^\prime \| (\mathcal P+\mathcal K)u\|_{C^{k-2,\alpha}_2(M)}
\end{equation}
for all functions $u\in C^{k,\alpha}_2(M)$.
Since $\mathcal P + \mathcal K$ is invertible, we may view \eqref{LinearizedLichRatio} as a fixed-point problem for the mapping 
\begin{equation*}
\mathcal G\colon w \mapsto (\mathcal P+ \mathcal K)^{-1}\left[f+Q(w)\right].
\end{equation*}

Fix $r_*\in (0,1)$.
Note that $Q(w) = \sum_i a_i u^2 F_i(w)$ for functions $a_i\in C^{k-2,\alpha}(M)$.
Because $F_i\colon (-1,\infty)\to\mathbb R$ is smooth,
there exists some constant $C_*$ such that if  $\|w\|_{C^{k,\alpha}_2(M)},\|v\|_{ C^{k,\alpha}_2(M)}\leq  r_*$ then 
\begin{align*}
\| F_i(v) \|_{C^{k-2,\alpha}(M)}
&\leq C_* \|v\|_{C^{k-2,\alpha}(M)},\\
\| F_i(w) - F_i(v)\|_{C^{k-2,\alpha}(M)}
&\leq C_* \|w-v\|_{C^{k-2,\alpha}(M)},
\end{align*}
and hence 
\begin{equation*}
\begin{aligned}
\| Q(w) &- Q(v)\|_{C^{k-2,\alpha}_2(M)}
\\
&\leq \sum_i \|a_i\|_{C^{k-2,\alpha}(M)} \| u^2F_i(w) - v^2F_i(v)\|_{C^{k-2,\alpha}_2(M)}
\\
&=\sum_i \|a_i\|_{C^{k-2,\alpha}(M)} \| (w^2 - v^2)F_i(v) + v^2(F_i(w)-F_i(v))\|_{C^{k-2,\alpha}_2(M)}\\
&\leq \sum_i \left(
C_* \|w-v\|_{C^{k-2,\alpha}_2(M)}
\|w+v\|_{C^{k-2,\alpha}(M)}
\right.
\\
&\qquad\qquad +
\left.
C_*\|v\|_{C^{k-2,\alpha}(M)}\|v\|_{C^{k-2,\alpha}_2(M)}
\|w-v\|_{C^{k-2,\alpha}(M)}
\right),
\end{aligned}
\end{equation*}
where $C_*$ now also depends on an upper bound for the norms of $\phi_1$, $A_2$ and $B_2$.

We now show that for sufficiently small $ r\in (0, r_*)$ the map $\mathcal G$ is a contraction from the ball $B_ r(0)$ in $C^{k,\alpha}_2(M)$ to itself.
Because $\|w\|_{C^{k,\alpha}(M)} = \|\rho^2 w\|_{C^{k,\alpha}_2(M)}$, there is a constant
$C_\rho\geq 1$ such that $\|w \|_{C^{k,\alpha}(M)}\leq C_\rho \|w\|_{C^{k,\alpha}_2(M)}$ for all $w\in C^{k,\alpha}_2(M)$.
Thus if $w,v\in B_r(0)$ we have
\begin{equation}
\label{Questimate}
\|Q(w) - Q(v)\|_{C^{k-2,\alpha}_2(M)}
\leq 3 C_* C_\rho^2 r \|w-v\|_{C^{k-2,\alpha}_2(M)}.
\end{equation}
Taking $v=0$ and using \eqref{UniformInversion} and \eqref{f-Small} we conclude that
\begin{equation*}
\begin{aligned}
\|\mathcal G(w)\|_{C^{k,\alpha}_2(M)}
&\leq C^\prime \| f + Q(w)\|_{C^{k-2,\alpha}_2(M)}
\\
&\leq C^\prime C\delta
+3 C^\prime C_* C_\rho^2 r^2.
\end{aligned}
\end{equation*}
Thus if we choose $r$ small enough that $3C'C_*C^2_\rho r<\frac12$ and then
$\delta$ small enough that $C'C\delta<r/2$,
it follows that $\mathcal G$ maps $B_r(0)$ to itself.
To see that $\mathcal G$ is a contraction for $r$ sufficiently small, we simply apply \eqref{UniformInversion} and 
\eqref{Questimate}
to $\mathcal G(w) - \mathcal G(v) = (\mathcal P + \mathcal K)(Q(w) - Q(v))$:
\begin{align*}
\|\mathcal G(w)-\mathcal G(v)\|_{C^{k,\alpha}_2(M)}
&\le C' \|Q(w)-Q(v)\|_{C^{k-2,\alpha}_2(M)}
\\
&\le 3 C' C_* C_\rho^2 r \|w-v\|_{C^{k-2,\alpha}_2(M)}.
\end{align*}

The Banach Fixed-Point Theorem then implies the existence of a unique $u^\prime\in B_r(0)$ satisfying $\mathcal G(u^\prime) = u^\prime$.
This gives rise to a positive and bounded function $\phi^\prime_2 = \phi_1(1+u^\prime)$ satisfying $\phi^\prime_2-1\in C^{k,\alpha}_2(M)$ and  \eqref{LichPhii} with $i=2$.
The uniqueness of solutions to \eqref{SimpleLichPhi} implies that $\phi_2 = \phi^\prime_2$ (i.e.,~that $u = u^\prime$).

Furthermore, we have
\begin{equation*}
\|\phi_1-\phi_2\|_{C^{k,\alpha}_2(M)} 
= \|\phi_1 u\|_{C^{k,\alpha}_2(M)}
\leq \|\phi_1\|_{C^{k,\alpha}(M)} \|u\|_{C^{k,\alpha}_2(M)}
\leq C_1r,
\end{equation*}
for some constant $C_1$  determined by $\lambda_1$, $A_1$, and $B_1$.
Thus given $\epsilon>0$, we just choose $r$ small enough that $C_1r<\epsilon$,
and then choose $\delta$ satisfying all of the constraints above.
\end{proof}

\section{Parametrization and existence of shear-free initial data}
\label{Section-PandE}

The results of Sections \ref{PDEResults} and \ref{PDE-Continuity} allow us to carry out the program outlined in \S\ref{Section-ConformalMethod} for parametrizing collections of initial data sets by collections of seed data, as well as showing that all shear-free seed data, in the regularity classes defined above, arises from free data.
We show that the maps taking the various regularity classes of free data sets to seed data sets, and the maps taking seed data sets to initial data sets, are continuous and thus give rise to homeomorphisms.

Continuity is established with respect to the  topologies induced from the normed spaces in which the various elements of each data set live.
For example, a sequence  $(g_i, K_i, \Phi_i)$ converges to $(g,K,\Phi)$ in $\mathscr D^{k,\alpha}$ if  $\bar g_i \to \bar g \in \mathscr C^{k,\alpha;2}(M)$, $\bar\Sigma_i \to \bar \Sigma \in \mathscr C^{k-1,\alpha;1}(M)$, and $\Phi_i \to \Phi$ in the space indicated in \eqref{MatterRegularity}.

We now address the parametrization of initial data in terms of seed data.
Proposition \ref{SolveLichPhi} implies that for each seed data set $( \lambda,\sigma,\Psi)$ we obtain a solution $\phi$ to equation \eqref{LichPhi}.
This gives rise to a map
\begin{equation}
\label{DefinePi}
\Pi\colon (\lambda, \sigma,\Psi)
\mapsto (g,K, \Phi) =  \(\phi^4\lambda, \phi^{-2}\sigma-g, \phi^2\odot\Psi),
\end{equation}
from which we obtain the desired parametrizations.

\begin{theorem}
\label{ParametrizationTheorem}
Let $k\geq 2$ and $\alpha\in (0,1)$.
The map $\Pi$ gives rise to a continuous surjection
\begin{equation*}
\Pi \colon \mathscr S^{k,\alpha} \to \mathscr D^{k,\alpha},
\end{equation*}
with right inverse $\iota\colon (g,K,\Phi) \mapsto (g,K+g,\Phi)$.
The map $\Pi$ descends to a homeomorphism 
\begin{equation}
\label{SeedHomeomorphisms}
\widetilde \Pi \colon \mathscr S^{k,\alpha}\slash{\sim} \to \mathscr D^{k,\alpha},
\end{equation}
where $\mathscr S^{k,\alpha}\slash{\sim}$ is the quotient of $\mathscr S^{k,\alpha}$ under the action of positive functions $\theta$ with $\theta -1\in C^{k,\alpha}_1(M)$ given by \eqref{ConformalCovarianceMapping}.

Furthermore, the map $\Pi$ restricts to surjective  maps $\mathscr S^{\infty} \to \mathscr D^{\infty}$ and $\mathscr S^{\phg} \to \mathscr D^{\phg}$.
These in turn descend to bijections $\mathscr S^{\infty}\slash{\sim} \to \mathscr D^{\infty}$ and $\mathscr S^{\phg}\slash{\sim} \to \mathscr D^{\phg}$, where $\mathscr S^\infty\slash{\sim}$ is the quotient of $\mathscr S^\infty$ under the action of positive functions $\theta$ with $\theta -1\in C^{\infty}_1(M)$, and $\mathscr S_\phg\slash{\sim}$ is the quotient of $\mathscr S_\phg$ under the action of positive functions $\theta\in C^2_\phg(\bar M)$ with $\theta-1= \mathcal O(\rho)$.

\end{theorem}

\begin{proof}
Consider a seed data set $(\lambda, \sigma, \Psi)\in \mathscr S^{k,\alpha}$.
Let $\theta \in \mathcal C^{k,\alpha;2}(M)$ be the function given by Proposition \ref{prop:DeformToHAR} such that the scalar curvature of $\tilde\lambda = \theta^4\lambda$ satisfies $\R[\tilde\lambda]+6\in C^{k-2,\alpha}_2(M)$.
Observe that $(\tilde\lambda, \tilde\sigma,\tilde\Psi) = (\theta^4\lambda, \theta^{-2}\sigma, \theta^2\odot\Psi)\in \mathscr S^{k,\alpha}$.
Proposition \ref{prop:DeformToHAR} further implies that the map $(\lambda, \sigma,\Psi) \mapsto (\theta^4\lambda, \theta^{-2}\sigma, \theta^2\odot\Psi)$ is locally Lipschitz continuous.

It follows from the definition of $\mathscr S^{k,\alpha}$ that the functions 
\begin{equation}
\label{DefineAandB}
\tilde A:=\frac18|\tilde\sigma|^2_{\tilde\lambda}
\quad\text{ and }\quad
\tilde B:= \frac18\left(|\tilde{\mathcal E}|^2_{\tilde\lambda} + |\tilde{\mathcal B}|^2_{\tilde\lambda} + 2\tilde\zeta\right)
\end{equation}
are in $C^{k-1,\alpha}_2(M)$.
Thus from Proposition \ref{SolveLichPhi}\eqref{part:simple-lich-better-r}
there exists a unique positive $\tilde\phi$ with $\tilde\phi-1\in C^{k,\alpha}_2(M)$ satisfying
\begin{equation*}
\Delta_{\tilde\lambda}\tilde\phi = \frac18R[\tilde\lambda]\phi - \tilde A\tilde \phi^{-7} - \tilde B\tilde \phi^{-3} + \frac34\tilde \phi^5
\end{equation*}
and such that $\tilde\phi^4\tilde\lambda \in \WAH^{k,\alpha;2}$.
From Proposition \ref{prop:LichPhiCont}, the map $(\tilde\lambda, \tilde\sigma, \tilde\Psi)\mapsto \tilde\phi$ is continuous.
Thus setting $g = \tilde\phi^{4}\tilde\lambda$, $K = \tilde\phi^{-2}\tilde\sigma -g$, and $\Phi = \tilde\phi\odot\tilde\Psi$ we have that
$(\tilde\lambda, \tilde\sigma, \tilde\Psi)\mapsto (g, K,\Phi)$ is continuous. 

We now invoke the conformal covariance of the Lichnerowicz equation as recorded in Lemma \ref{lemma:LichInvariance} to conclude that $\phi = \theta\tilde\phi$ is the unique solution to \eqref{LichPhi} given by Proposition \ref{SolveLichPhi}\eqref{part:FirstSolveLichPhi} and thus $(g,K,\Phi) = \Pi(\lambda,\sigma,\Psi)$; furthermore the map $\Pi$ is continuous.

Since $\Pi$ has a continuous right inverse, it is surjective and is a quotient map.
The invariance of $\Pi$ and \eqref{LichPhi} under the transformation \eqref{ConformalCovarianceMapping} implies that $\Pi$ descends maps $\widetilde\Pi$ as in \eqref{SeedHomeomorphisms}.
It is straightforward to verify that for any $(g,K,\Phi)$, the preimage $\Pi^{-1}(g,K,\Phi)$ consists of a single equivalence class and thus $\widetilde\Pi$ provides a homeomorphism as claimed.

Finally, the regularity results of Proposition \ref{SolveLichPhi} further imply that $\Pi$ takes $\mathscr S^\infty$ to $\mathscr D^\infty$ and $\mathscr S_\phg$ to $\mathscr D_\phg$.
\end{proof}

We now show that each free data set in $\mathscr F^{k,\alpha}$, $\mathscr F^\infty$, or $\mathscr F_\phg$ projects continuously to a seed data set in $\mathscr S^{k,\alpha}$, $\mathscr S^\infty$,  or $\mathscr S_\phg$, respectively, and that every set of seed data can be obtained in this way.
To accomplish this, we define a map 
\begin{equation}
\Xi\colon (\lambda,\nu,\Upsilon) \mapsto (\lambda,\sigma,\Psi)
\end{equation}
based on the procedure described in \S\ref{Section-ConformalMethod}, which we now recall. If $(\lambda,\nu,\Upsilon)$ is a free data set with $\Upsilon = (e, b, j, \zeta)$, we apply  Corollary \ref{Helmholtz} to  obtain unique functions $u$ and $v$ such that the vector fields $\mathcal E := e-\grad_\lambda u$ and $\mathcal B := b-\grad_\lambda v$ are divergence free with respect to $\lambda$.
We subsequently set $\Psi = (\mathcal E, \mathcal B, j, \zeta)$.

We next invoke Proposition \ref{SolveVectorLaplacian} to obtain a vector field $W$ such that
\begin{equation}
\label{E-VL}
L_\lambda W = \Div_\lambda(\rho^{-1}\B_{\bar\lambda}(\rho)+\nu)^\sharp - j - \mathcal E\times_\lambda\mathcal B.
\end{equation}
We set $\sigma = \rho^{-1}\B_{\bar\lambda}(\rho)+\nu + \mathcal D_\lambda W$.

\begin{theorem}
\label{ExistenceProjection}
Fix $k\geq 3$ and $\alpha \in (0,1)$. 
Then the map $\Xi$ gives rise to a continuous projection $\Xi\colon \mathscr F^{k,\alpha} \to \mathscr S^{k,\alpha}$.
It furthermore restricts to projections $\mathscr F^\infty \to \mathscr S^\infty$ and $\mathscr F_\phg \to \mathscr S_\phg$.
\end{theorem}

\begin{proof}
Suppose $(\lambda, \nu, \Upsilon)\in \mathscr F^{k,\alpha}$ is a free data set with $\Upsilon = (e,b,j,\zeta)$.
Corollary \ref{Helmholtz} ensures that there exist functions $u$ and $v$ in $ C^{k-1,\alpha}_1(M)$ satisfying \eqref{LaplaceUV} and such that $\mathcal E = e - \grad_\lambda u$ and  $\mathcal B= \grad_\lambda v$ are divergence-free vector fields in $C^{k-2,\alpha}_1(M)$.
Thus $\Psi = (\mathcal E, \mathcal B, j, \zeta)$ has the regularity of \eqref{MatterRegularity},  as required by the definition of $\mathscr S^{k,\alpha}$.
Since divergence is a geometric operator, Propositions \ref{prop:LinearContinuity} and \ref{prop:GeoOpCont} imply that $\Upsilon \mapsto \Psi$ is locally Lipschitz continuous.

Proposition \ref{B-BasicProperties}\eqref{B-TransverseProperty} implies that
\begin{equation}
\label{DivergenceOfMu}
\begin{aligned}
\Div_\lambda (\rho^{-1}\B_{\bar\lambda}(\rho)+\nu) 
&= \rho^{-1}\Div_\lambda \B_{\bar\lambda}(\rho) 
+ \Div_\lambda\nu
\\
&= \rho
\Div_{\bar\lambda}\B_{\bar\lambda}(\rho)
+ \Div_\lambda\nu,
\end{aligned}
\end{equation}
and thus Lemma \ref{B-in-WAH} implies that $\Div_\lambda (\rho^{-1}\B_{\bar\lambda}(\rho)+\nu) \in C^{k-2,\alpha}_2(M)$.
Consequently, we obtain from Proposition \ref{SolveVectorLaplacian} a vector field $W\in C^{k,\alpha}_2(M)$ satisfying \eqref{E-VL}.

Since $W\in C^{k,\alpha}_2(M)$ we have $\mathcal D_\lambda W\in C^{k-1,\alpha}_2(M)$.
Thus, setting 
$$\sigma
=  \rho^{-1}\B_{\bar \lambda}(\rho) + \nu +\mathcal D_{\lambda} W,$$ 
we see that \eqref{sigmaBC} is satisfied and the map $\Xi$ takes free data sets in $\mathscr F^{k,\alpha}$ to seed data sets in $\mathscr S^{k,\alpha}$ as claimed.
It follows immediately that $\Xi$ also maps $\mathscr S^\infty$ to $\mathscr S^\infty$.

Lemma \ref{lemma:cont-hess-etc} provides the continuity of $\lambda\mapsto \B_{\bar\lambda}(\rho)$ and of $\lambda\mapsto \Div_{\lambda}(\rho^{-1}\B_{\bar\lambda}(\rho))$.
Thus, as $\mathcal D_\lambda$ is a geometric operator, Propositions \ref{prop:LinearContinuity} and \ref{prop:GeoOpCont} imply that $(\lambda,\nu)\mapsto \sigma$ is locally Lipschitz continuous.

In the case of polyhomogeneous seed data, we note that the regularity results of Corollary \ref{Helmholtz} and Proposition \ref{SolveVectorLaplacian} give rise to polyhomogeneous solutions, and thus $\Xi$ indeed maps $\mathscr F_\phg$ to $\mathscr S_\phg$.

We conclude the proof by noting that
\begin{equation}
\iota\colon(\lambda, \sigma, \Psi)
\mapsto
(\lambda, \sigma- \rho^{-1}\B_{\lambda}(\rho), \Psi)
\end{equation}
 satisfies $\Xi\circ\iota = \Id$ and thus the maps are indeed projections.
\end{proof}

\section{Weakly asymptotically hyperbolic solutions}
\label{Section-WeakExistence}

The previous section describes how the PDE results of \S\ref{PDEResults} may be used to obtain shear-free initial data, provided the free metric is in $\WAH^{k,\alpha;2}$.
The same sequence of PDE results can also be used to obtain weakly asymptotically hyperbolic solutions to the constraint equations with less boundary regularity.
In particular, we may construct solutions to the constraint equations that need not be $C^1$ conformally compact, and may not be sufficiently regular to make sense of the shear-free condition.

\begin{theorem}
\label{WeakExistenceTheorem}
Suppose $\lambda  \in \WAH^{k,\alpha;1}$ with $k\geq 3$ and $\alpha \in (0,1)$.
Let $ \nu\in C^{k-1,\alpha}_1(M)$ be a symmetric covariant $2$-tensor field that is traceless with respect to $\lambda$, and let $\Upsilon = (e,b,j,\zeta) \in C^{k-2,\alpha}_1(M)$ be a set of matter fields.

Then there exist uniquely defined functions $u,v\in C^{k-1,\alpha}_1(M)$, a uniquely defined vector field $W\in C^{k,\alpha}_1(M)$, and a uniquely defined and positive function $\phi$ with $\phi-1\in C^{k,\alpha}_1(M)$ such that
\begin{equation}
\begin{gathered}
g := \phi^4 \lambda 
\in \WAH^{k,\alpha;1},
\\
\Sigma := \phi^{-2}(\rho^{-1}\B_{\rho^2\lambda}(\rho) + \nu + \mathcal D_\lambda W) 
\in C^{k-1,\alpha}_1(M),
\\
\Phi := \phi^2\odot \left( \Upsilon - (\grad_{\lambda}u, \grad_{\lambda}v, 0, 0) \right)
\in C^{k-2,\alpha}_1(M),
\end{gathered}
\end{equation}
give rise to a solution $(g,K,\Phi) = (g,\Sigma - g,\Phi)$ to the constraint equations \eqref{HamiltonianConstraint}--\eqref{MomentumConstraint}--\eqref{MaxwellConstraints} on $M$.
\end{theorem}

\begin{proof}
Applying Corollary \ref{Helmholtz} to $e$ and $b$, we obtain functions $u,v\in C^{k-1,\alpha}_1(M)$ such that $\mathcal E = e-\grad_\lambda u$ and $\mathcal B = b-\grad_\lambda v$ are divergence-free vector fields in $C^{k-2,\alpha}_1(M)$.

From Proposition \ref{B-in-WAH} we have $\B_{\rho^2\lambda}(\rho)\in C^{k-1,\alpha}_2(M)$, and from  Proposition \ref{B-BasicProperties}\eqref{B-TransverseProperty} we have
\begin{equation}
\Div_\lambda [\rho^{-1}\B_{\rho^2\lambda}(\rho)] = \rho^{-1}\Div_\lambda [\B_{\rho^2\lambda}(\rho)]\in C^{k-2,\alpha}_1(M).
\end{equation}
Thus the vector field $Y := \Div_\lambda [\rho^{-1}\B_{\rho^2\lambda}(\rho) + \bar \nu]^\sharp -  j - \mathcal E \times_\lambda \mathcal B$ is an element of $C^{k-2,\alpha}_1(M)$.
By Proposition \ref{SolveVectorLaplacian} there exists a unique vector field $W\in C^{k,\alpha}_1(M)$ such that $L_\lambda W= Y$.

Let $\sigma = \rho^{-1}\B_{\rho^2\lambda}(\rho) + \bar\nu + \mathcal D_\lambda W$.
Note that both 
\begin{equation}
A=\frac18|\sigma|^2_\lambda
\quad\text{ and }\quad
B=\frac18\left(|\mathcal E|^2_\lambda + |\mathcal B|^2_\lambda + 2\zeta \right)
\end{equation}
are functions in $C^{k-1,\alpha}_{1}(M)$.
Proposition \ref{SolveLichPhi} thus guarantees the existence of a unique positive function $\phi$ with $\phi-1\in C^{k,\alpha}_1(M)$ that satisfies \eqref{LichPhi}.
\end{proof}

We remark that the continuity of the solution map defined by Theorem \ref{WeakExistenceTheorem} can be established by arguments analogous to, but simpler than, those used above.

We  emphasize that the boundary regularity required of  $\lambda$ in Theorem \ref{WeakExistenceTheorem} is significantly weaker than in the shear-free setting. 
Whereas in the shear-free setting the metric $\bar\lambda = \rho^2\lambda$ extends to a tensor field of class $C^{1,1}$ on $\bar M$, in the present setting the metric $\bar\lambda$ extends to a $C^{0,1}$, but not necessarily $C^1$, tensor field on $\bar M$.
Consequently, Theorem \ref{WeakExistenceTheorem} is an  improvement of previous existence theorems, such as those in \cite{AnderssonChrusciel-Dissertationes}, where it is assumed that the ``background'' metric $\bar\lambda$ is $C^2$ on $\bar M$.

Likewise, the tensor field $ \nu$, which in the shear-free setting was required to be in $L^\infty(\bar M)$, may not even be pointwise bounded with respect to  the background metric $\bar h$ under the hypotheses of Theorem \ref{WeakExistenceTheorem}.

\bibliographystyle{plain}
\bibliography{AHEM}

\end{document}